\documentclass[a4paper]{amsart}                

\usepackage[latin1]{inputenc}                  
\usepackage[T1]{fontenc}                       
\usepackage[spanish,english]{babel}            
\usepackage{amsmath,amssymb,amsthm}            
\usepackage{latexsym}                          
\usepackage{delarray}                          
\usepackage{bbm}                               
\usepackage{hyperref}                          
\usepackage[pdftex,usenames,dvipsnames]{color} 
\usepackage{bbding,trfsigns}                   
\usepackage{wasysym}                           
\usepackage{datetime}                          

\usepackage{tikz}                   

\DeclareMathAlphabet{\mathpzc}{OT1}{pzc}{m}{it} 

\newtheorem{Th}{Theorem}[section]              
\newtheorem{Cor}{Corollary}[section]
\newtheorem{Rem}{Remark}[section]

\newtheorem{Prop}[Th]{Proposition}
\newtheorem{Lem}{Lemma}[section]

\newcommand{\C}{\mathbb{C}}
\newcommand{\N}{\mathbb{N}}

\newcommand{\R}{{\mathbb{R}}}

\DeclareMathOperator{\supp}{supp}

\title[Variable exponent Sobolev spaces associated with Jacobi expansions]
      {Variable exponent Sobolev spaces associated with Jacobi expansions}

\author[V. Almeida]{V. Almeida}
\author[J.J. Betancor]{J.J. Betancor}
\author[A.J. Castro]{A.J. Castro}
\author[A. Sanabria]{A. Sanabria}
\author[R. Scotto]{R. Scotto}

\address{\newline
        V\'{\i}ctor Almeida, Jorge J. Betancor, Alejandro Sanabria \newline
        Departamento de An\'alisis Matem\'atico,
        Universidad de la Laguna, \newline
        Campus de Anchieta, Avda. Astrof\'{\i}sico Francisco S\'anchez, s/n, \newline
        38271, La Laguna (Sta. Cruz de Tenerife), Spain}
\email{valmeida@ull.es, jbetanco@ull.es, asgarcia@ull.es}

\address{\newline
        Alejandro J. Castro\newline
        Department of  Mathematics,
       Uppsala University, \newline
       S-751 06 Uppsala, Sweden}
\email{alejandro.castro@math.uu.se}

\address{\newline
       Roberto Scotto \newline
       IMAL- Facultad de Ingenier\'{\i}a Qu\'{\i}mica, U.N. del Litoral,\newline
        Santiago del Estero 2829, Santa Fe 3000, Argentina}
\email{roberto.scotto@gmail.com}



\begin{document}

  \footnotetext{Date: \today.}

  \maketitle                

\begin{abstract}
    In this paper we define variable exponent Sobolev spaces associated with Jacobi expansions.
    We prove that our generalized Sobolev spaces can be characte\-rized as variable exponent potential spaces and as variable
    exponent Triebel-Lizorkin type spaces.
\end{abstract}

\section{Introduction} \label{sec:intro}

Sobolev spaces associated with orthogonal systems have been studied in the last years. Bongioanni and Torrea (\cite{BT1} and \cite{BT2})
defined Sobolev spaces in the Hermite and Laguerre settings. Sobolev spaces associated with ultraspherical expansions were investigated
by Betancor, Fari\~na, Rodr\'{\i}guez-Mesa, Testoni and Torrea \cite{BFRTT1}. The study in \cite{BFRTT1} was extended recently to Jacobi
expansions by Langowski \cite{La}.

In this paper we define variable exponent Sobolev spaces in the  Jacobi context. We now describe our main results.

Consider a measurable function $p:\Omega\subseteq{\Bbb R}^n \longrightarrow [1,\infty)$. By $L^{p(\cdot)}(\Omega)$ we denote the variable
exponent Lebesgue space that consists of all those measurable functions on $\Omega$ such that for some $\lambda >0$
$$\int_{\Omega}{\left({{|f(x)|}\over{\lambda}}\right)^{p(x)}dx}
    <\infty.$$
It is a Banach space with the Luxermburg norm defined by
$$\|f\|_{L^{p(\cdot)}(\Omega)}
    =\inf\left\{\lambda>0\ : \ \int_{\Omega}{\left({{|f(x)|}\over{\lambda}}\right)^{p(x)}dx}\leq 1\right\}, \quad f\in L^{p(\cdot)}(\Omega).$$
 By $p'(\cdot)$ we represent the conjugate variable exponent.
A complete study of $L^{p(\cdot)}$-spaces can be found in \cite{DHHR}.

We define ${\mathcal P}(\Omega)$ as the set of measurable functions $p:\ \Omega\ \longrightarrow\ [1,\infty)$ such that
$$p_-=\mathrm{ess}\, \inf\{p(x)\ : \ x\in\Omega\}>1
\qquad \mbox{and} \qquad
p_+=\mathrm{ess}\, \sup\{p(x)\ : \ x\in\Omega\}<\infty.$$
The  Hardy-Littlewood maximal operator ${\mathcal M}$ is defined as
$${\mathcal M}f(x)=\sup_{B\ni x}{1\over{|B|}}\int_{B}{f(y)dy}, \quad x\in\Omega.$$
The set $B$ in the supremum represents a ball and $|B|$ denotes its Lebesgue measure.

We define ${\mathcal B}(\Omega)$ as the subset of ${\mathcal P}(\Omega)$ that consists of all those measurable
functions $p$ such that the maximal operator  ${\mathcal M}$ is bounded from
$L^{p(\cdot)}(\Omega)$ into itself. Diening \cite[Theorem 3.5]{Die1} proved that if $\Omega$ is a bounded subset of
${\Bbb R}^n$, $p\in{\mathcal P}(\Omega)$ and there exists $C>0$ such that
\begin{equation}\label{eq:log}
    |p(x)-p(y)|\leq{C\over{-\log|x-y|}}, \quad x,y\in\Omega,\ \ |x-y|\leq 1/2,
\end{equation}
then $p\in{\mathcal B}(\Omega)$.

Many classical operators in harmonic analysis (maximal operator, singular integrals, Fourier multipliers,
commutators, fractional integrals, ...) have been studied in variable  $L^{p(\cdot)}$-spaces (see, for instance,
\cite{CFMP}, \cite{DHHR}, \cite{DR} and \cite{Sa}).

Let $k\in\Bbb N$, where by $\Bbb N$ we represent the set of positive integer with zero included,  and $p\in{\mathcal P}(\Omega)$. A measurable function $f$ on $\Omega$ is in the generalized
Sobolev space $W^{k,p(\cdot)}(\Omega)$ if its weak partial derivatives $D^{\alpha}f\in L^{p(\cdot)}(\Omega)$,
$\alpha\in{\Bbb N}^n$ and $0\leq|\alpha|\leq k$. The norm in $W^{k,p(\cdot)}(\Omega)$ is defined by
$$\|f\|_{W^{k,p(\cdot)}(\Omega)}
    =\sum_{|\alpha|=0}^k{\|D^{\alpha}f\|_{L^{p(\cdot)}(\Omega)}}, \quad f\in W^{k,p(\cdot)}(\Omega).$$
It turns out that $W^{k,p(\cdot)}(\Omega)$ is a Banach space.

Variable exponent Sobolev spaces $W^{k,p(\cdot)}(\Omega)$ have been studied by a lot of authors in this century. Applications of these
generalized Sobolev spaces can be seen in \cite[Part III]{DHHR}.

Now we turn to the Harmonic Analysis associated with the Jacobi differential operator $L_{\alpha,\beta}$ for $\alpha,\beta >-1$, which is defined as
$$L_{\alpha,\beta}
    =-{{d^2}\over{d{\theta}^2}}-{{1-4{\alpha}^2}\over{16{\sin}^2\frac{\theta}{2}}}-{{1-4{\beta}^2}\over{16{\cos}^2\frac{\theta}{2}}},  \quad \mbox{on}\ (0,\pi).$$
This type of analysis  has emerged as a prolific area of interest
(see \cite{BU}, \cite{CNZ}, \cite{CNS}, \cite{La}, \cite{Li}, \cite{LSL}, \cite{NR}, \cite{NSjSz} and \cite{Stem3},  amongst others).

The Jacobi operator admits the following decomposition
$$L_{\alpha,\beta}
    =D_{\alpha,\beta}^*D_{\alpha,\beta}+\left({{\alpha+\beta+1}\over 2}\right)^2,$$
where
\begin{align*}
    D_{\alpha,\beta}
        & ={d\over {d\theta}}-{{2\alpha+1}\over 4}\cot\frac{\theta}{2}+{{2\beta+1}\over 4}\tan\frac{\theta}{2} \\
        & = \Big( \sin \frac{\theta}{2} \Big)^{\alpha + 1/2} \Big( \cos \frac{\theta}{2} \Big)^{\beta + 1/2}
            \frac{d}{d\theta} \Big[ \Big(\sin \frac{\theta}{2} \Big)^{-\alpha - 1/2} \Big( \cos \frac{\theta}{2} \Big)^{-\beta - 1/2} \Big],
\end{align*}
and $D_{\alpha,\beta}^*$ is the formal adjoint of $D_{\alpha,\beta}$ in $L^2(0,\pi)$. When $\alpha=\beta$
the Jacobi operator $L_{\alpha,\beta}$ reduces to the ultraspherical operator $L_{\lambda}$, $\lambda =\alpha+1/2$,
considered in \cite{BFRTT1}. According to \cite[(4.24.2)]{Sz} we have that, for every $n\in\Bbb N$,
$$L_{\alpha,\beta}{\phi}_n^{\alpha,\beta}
    ={\lambda}_n^{\alpha,\beta}{\phi}_n^{\alpha,\beta},$$
where ${\lambda}_n^{\alpha,\beta}=(n+{{\alpha+\beta+1}\over 2})^2$ and
$${\phi}_n^{\alpha,\beta}(\theta)
    =\Big(\sin {{\theta}\over 2}\Big)^{\alpha+1/2}\Big(\cos {{\theta}\over 2}\Big)^{\beta+1/2}\mathcal{P}_n^{\alpha,\beta}(\theta), \quad \theta\in(0,\pi).$$
If $p_n^{\alpha,\beta}$ denotes the $n$-th Jacobi polynomial considered in Szeg\"o's monograph, then
$\mathcal{P}_n^{\alpha,\beta}=d_n^{\alpha,\beta}p_n^{\alpha,\beta}$, where $d_n^{\alpha,\beta}$ is a normalization constant,
for every $n\in\Bbb N$. The system $\{{\phi}_n^{\alpha,\beta}\}_{n\in\Bbb N}$ is orthonormal and complete in $L^2(0,\pi)$.
We define the Jacobi operator ${\mathcal L}_{\alpha,\beta}$ by
$${\mathcal L}_{\alpha,\beta}f
    =\sum_{n=0}^{\infty}{{\lambda}_n^{\alpha,\beta}c_n^{\alpha,\beta}(f){\phi}_n^{\alpha,\beta}},  \quad f\in D({\mathcal L}_{\alpha,\beta}).$$
Here, for every $f\in L^2(0,\pi)$ and $n\in\Bbb N$,
$$c_n^{\alpha,\beta}(f)
    =\int_0^{\pi}{{\phi}_n^{\alpha,\beta}(\theta)f(\theta)d\theta},$$
and by $D({\mathcal L}_{\alpha,\beta})$ we denote the domain of ${\mathcal L}_{\alpha,\beta}$ given by
$$D({\mathcal L}_{\alpha,\beta})
    =\{f\in L^2(0,\pi)\ : \ \sum_{n=0}^{\infty}{({\lambda}_n^{\alpha,\beta})^2|c_n^{\alpha,\beta}(f)|^2}<\infty\}.$$
Note that $C_c^{\infty}(0,\pi)$, the space of smooth function with compact support in $(0,\pi)$, is contained in
$D({\mathcal L}_{\alpha,\beta})$ and hence,
$${\mathcal L}_{\alpha,\beta}f
    =L_{\alpha,\beta}f, \quad f\in C_c^{\infty}(0,\pi).$$
${\mathcal L}_{\alpha,\beta}$ is a positive and selfadjoint operator in $L^2(0,\pi)$. Let us note that $-{\mathcal L}_{\alpha,\beta}$
generates a semigroup of operators $\{W_t^{\alpha,\beta}\}_{t>0}$ in $L^2(0,\pi)$ where, for every $t>0$,
$$W_t^{\alpha,\beta}f
    =\sum_{n=0}^{\infty}{e^{-t{\lambda}_n^{\alpha,\beta}}c_n^{\alpha,\beta}(f){\phi}_n^{\alpha,\beta}}, \quad f\in L^2(0,\pi).$$
Moreover, for every $t>0$ and $f\in L^2(0,\pi)$,
$$W_t^{\alpha,\beta}f(\theta)
    =\int_0^{\pi}{W_t^{\alpha,\beta}(\theta,\varphi)f(\varphi)d\varphi}, \quad \theta\in (0,\pi),$$
where
$$W_t^{\alpha,\beta}(\theta,\varphi)
    =\sum_{n=0}^{\infty}{e^{-t{\lambda}_n^{\alpha,\beta}}{\phi}_n^{\alpha,\beta}(\theta){\phi}_n^{\alpha,\beta}(\varphi)}, \quad \theta,\varphi\in (0,\pi)\ \mbox{and}\ t>0.$$
$\{W_t^{\alpha,\beta}\}_{t>0}$ is called the heat semigroup associated with the Jacobi operator ${\mathcal L}_{\alpha,\beta}$.
By $\{P_t^{\alpha,\beta}\}_{t>0}$ we denote the Poisson semigroup defined by ${\mathcal L}_{\alpha,\beta}$.
According to the subordination formula, we can write, for every $t>0$ and $f\in L^2(0,\pi)$,
$$ P_t^{\alpha,\beta}f(\theta)
     =\int_0^{\pi}{P_t^{\alpha,\beta}(\theta,\varphi)f(\varphi)d\varphi}, \quad \theta\in (0,\pi),$$
where
\begin{equation}\label{eq:3.1.1}
    P_t^{\alpha,\beta}(\theta,\varphi)
        ={t\over{\sqrt{4\pi}}}\int_0^{\infty}{{{e^{-t^2/4u}}\over{u^{3/2}}}W_u^{\alpha,\beta}(\theta,\varphi)du}, \quad \theta,\varphi\in (0,\pi).
\end{equation}

Jacobi Sobolev spaces were studied by Langowski \cite{La}. We now introduce variable exponent Jacobi Sobolev spaces.
Assume that $p\in{\mathcal P}(0,\pi)$ and $k\in\Bbb N$. We say that a measurable function $f\in L^{p(\cdot)}(0,\pi)$
is in the variable Jacobi Sobolev space $W_{\alpha,\beta}^{k,p(\cdot)}(0,\pi)$ if ${\mathbb{D}}_{\alpha,\beta}^\ell f\in L^{p(\cdot)}(0,\pi)$,
for every $\ell\in\Bbb N$, $0\leq\ell\leq k$, with  ${\mathbb{D}}_{\alpha,\beta}^0 f=f$ and for $\ell\geq 1$,
$${\mathbb{D}}_{\alpha,\beta}^\ell =D_{\alpha+l-1,\beta+l-1}\circ\ ...\ \circ D_{\alpha+1,\beta+1}\circ D_{\alpha,\beta},$$
is understood in a weak sense. On $W_{\alpha,\beta}^{k,p(\cdot)}(0,\pi)$ we consider the norm defined by
$$\|f\|_{W_{\alpha,\beta}^{k,p(\cdot)}(0,\pi)}
    =\|f\|_{L^{p(\cdot)}(0,\pi)} + \sum_{l=1}^k{\|{\mathbb{D}}_{\alpha,\beta}^\ell f\|_{L^{p(\cdot)}(0,\pi)}}, \quad f\in W_{\alpha,\beta}^{k,p(\cdot)}(0,\pi).$$
Thus, $W_{\alpha,\beta}^{k,p(\cdot)}(0,\pi)$ becomes a Banach space.
See the discussion in \cite{La} (and also in \cite{BFRTT1}) for the use of the derivatives ${\mathbb{D}}_{\alpha,\beta}^\ell $,
instead of the more natural choice $D_{\alpha,\beta}^\ell=D_{\alpha,\beta}\circ\ ...\ \circ D_{\alpha,\beta}$.

Let $\gamma>0$ and assume that $\alpha+\beta \neq -1$. The negative power
${\mathcal L}_{\alpha,\beta}^{-\gamma}$ of ${\mathcal L}_{\alpha,\beta}$ is given by
\begin{equation}\label{eq:negpow}
    {\mathcal L}_{\alpha,\beta}^{-\gamma}f
            =\sum_{n=0}^{\infty}{(\lambda_n^{\alpha,\beta})^{-\gamma}c_n^{\alpha,\beta}(f){\phi}_n^{\alpha,\beta}}, \quad f\in L^2(0,\pi).
\end{equation}
 ${\mathcal L}_{\alpha,\beta}^{-\gamma}$ defines a one to one and bounded operator from
 $L^{p(\cdot)}(0,\pi)$ into itself (see Propositions \ref{propM3} and \ref{propM4} below).
 The variable exponent Jacobi potential space $H_{\alpha,\beta}^{\gamma,p(\cdot)}(0,\pi)$
 consists of all those functions  $f\in L^{p(\cdot)}(0,\pi)$  such that $f={\mathcal L}_{\alpha,\beta}^{-\gamma}g$ for some (unique)
 $g\in L^{p(\cdot)}(0,\pi)$. We considerer in $H_{\alpha,\beta}^{\gamma,p(\cdot)}(0,\pi)$ the following norm
$$\|f\|_{H_{\alpha,\beta}^{\gamma,p(\cdot)}(0,\pi)}
    =\|g\|_{L^{p(\cdot)}(0,\pi)}, \quad f={\mathcal L}_{\alpha,\beta}^{-\gamma}g\in H_{\alpha,\beta}^{\gamma,p(\cdot)}(0,\pi).$$
Endowed with this norm $H_{\alpha,\beta}^{\gamma,p(\cdot)}(0,\pi)$  is a Banach space.

The variable exponent version of \cite[Theorem A]{La} is given in the following theorem.

\begin{Th}\label{Th1}
    Let $\alpha, \beta \geq -1/2$ such that $\alpha+\beta \neq -1$ and $k\in\Bbb N$, $k\geq 1$.
    Assume that $p\in{\mathcal B}(0,\pi)$.
    Then, $H_{\alpha,\beta}^{k/2,p(\cdot)}(0,\pi)=W_{\alpha,\beta}^{k,p(\cdot)}(0,\pi)$. Moreover, the norms
    $\|\cdot\|_{H_{\alpha,\beta}^{k/2,p(\cdot)}(0,\pi)}$ and $\|\cdot\|_{W_{\alpha,\beta}^{k,p(\cdot)}(0,\pi)}$ are equivalent.
\end{Th}

The proof of Theorem \ref{Th1} is done in several steps. For a suitable function $p$ we will prove.
\begin{enumerate}
\item[(a)] The linear subspace $S_{\alpha,\beta}=\text{span}\{\phi_n^{\alpha,\beta}\}$ is dense in both $W_{\alpha,\beta}^{k,p(\cdot)}(0,\pi)$ and $H_{\alpha,\beta}^{k/2,p(\cdot)}(0,\pi)$.
\item[(b)]  The higher order Jacobi-Riesz transforms defined by
\begin{equation}\label{eq:3.1}
    R_{\alpha,\beta}^k
        ={\mathbb{D}}_{\alpha,\beta}^k{\mathcal L}_{\alpha,\beta}^{-k/2}\ \ \mbox{and}\ \ R_{\alpha,\beta}^{k,*}
        ={\mathbb{D}}_{\alpha,\beta}^{k,*}{\mathcal L}_{\alpha+k,\beta+k}^{-k/2}, \quad k\in\Bbb N.
\end{equation}
are bounded operators on $L^{p(\cdot)}(0,\pi)$.
\item[(c)] We define a multiplier operator $m(\mathcal{L}_{\alpha,\beta})$ in such a way that
 $$m({\mathcal L}_{\alpha,\beta}) R_{\alpha,\beta}^{k,*}R_{\alpha,\beta}^k f
 	= f - \sum_{n=0}^{k-1} c_n^{\alpha, \beta}(f) \phi_n^{\alpha,\beta}
	,\;\;\;\mbox{ for all}\; f\in S_{\alpha,\beta},$$
and prove its boundedness on $L^{p(\cdot)}(0,\pi)$.
\item[(d)] For every $\gamma >0$, the potential operator ${\mathcal L}_{\alpha,\beta}^{-\gamma}$ is also bounded on $L^{p(\cdot)}(0,\pi)$.
\end{enumerate}

According with \cite{CFMP} in order to get the boundedness of operators defined on $L^{p(\cdot)}(0,\pi)$ it is sufficient to prove boundedness of them on the weighted $L^r$-spaces, $L^r_{\omega}(0,\pi)$ for every $\omega\in A_r(0,\pi)$, the class of Muckenhoupt weights , and some $1<r<\infty$. Let us note that, taking into account \cite[Theorem 1.2]{CFMP}, we can change the condition "$p\in{\mathcal P}(0,\pi)$ and for some $p_0\in(1,p_-)$, $(p(\cdot)/p_0)'\in{\mathcal B}(0,\pi)$" used in \cite[Theorem 1.3]{CFMP} by $p\in{\mathcal B}(0,\pi)$, because if $p\in{\mathcal B}(0,\pi)$ there exists an extension $\widetilde{p}\in{\mathcal B}(\Bbb R)$ of $p$ from $(0,\pi)$ to $\Bbb R$.

Once of all this has been proved, the proof of Theorem \ref{Th1} is as follows:

From assertion (a) it is enough to prove the equivalence of norms for functions in $S_{\alpha,\beta}$. Let us take then $f,g\in S_{\alpha,\beta}$ such that $f={\mathcal L}_{\alpha,\beta}^{-k/2}g$. From assertions (b) and (c) we get
 \begin{align*}
        \|g\|_{L^{p(\cdot)}(0,\pi)}
            \leq & C \Big(\|m({\mathcal L}_{\alpha,\beta})R_{\alpha,\beta}^{k,*}{\mathbb{D}}_{\alpha,\beta}^kf\|_{L^{p(\cdot)}(0,\pi)}  + \|f\|_{L^{p(\cdot)}(0,\pi)} \Big) \\
         \leq & C \Big( \|{\mathbb{D}}_{\alpha,\beta}^kf\|_{L^{p(\cdot)}(0,\pi)} + \|f\|_{L^{p(\cdot)}(0,\pi)}\Big).
    \end{align*}
 Thus, we obtain
    $$\|f\|_{H^{k/2,p(\cdot)}_{\alpha,\beta}(0,\pi)}
        \leq C\|f\|_{W^{k,p(\cdot)}_{\alpha,\beta}(0,\pi)}.$$

On the other hand, by using assertions (b) and (d),  for every $m\in\Bbb N$ such that $0\leq m\leq k$,
    \begin{align*}
        \|{\mathbb{D}}_{\alpha,\beta}^mf\|_{L^{p(\cdot)}(0,\pi)}
            = & \|{\mathbb{D}}_{\alpha,\beta}^m{\mathcal L}_{\alpha,\beta}^{-k/2}g\|_{L^{p(\cdot)}(0,\pi)}
         = \|R_{\alpha,\beta}^m{\mathcal L}_{\alpha,\beta}^{-(k-m)/2}g\|_{L^{p(\cdot)}(0,\pi)}\leq C\|g\|_{L^{p(\cdot)}(0,\pi)}.
    \end{align*}
    Hence,
    $$\|f\|_{W^{k,p(\cdot)}_{\alpha,\beta}(0,\pi)}
        \leq C\|f\|_{H^{k/2,p(\cdot)}_{\alpha,\beta}(0,\pi)}.$$

We now define the positive power of the Jacobi operator
${\mathcal L}_{\alpha,\beta}$ according to the ideas of Lions and Peetre \cite[Chapter VII, Section 2]{LP} and Berens, Butzer
and Westphal \cite{BBW}. Let $\gamma >0$ and choose $r\in\Bbb N$
such that $\gamma<r\leq\gamma+1$. For every $\varepsilon>0$ and $f\in  L^{p(\cdot)}(0,\pi)$, we define
\begin{equation}\label{eqH1}
    I_{\varepsilon}^{\gamma,r}f
        =C_{\gamma,r} \int_{\varepsilon}^{\infty}{{\left(I-W_u^{\alpha,\beta}\right)^r f}\over{u^{\gamma+1}}}du\ ,
\end{equation}
where the integral is understood in the $ L^{p(\cdot)}$-Bochner sense and
$\displaystyle C_{\gamma,r}
    =\left(\int_{0}^{\infty}{{(1-e^{-u})^r}\over{u^{\gamma+1}}}du\right)^{-1}$. Note that, for every $f\in  L^{p(\cdot)}(0,\pi)$,
$$\int_{\varepsilon}^{\infty} \frac{\|\left(I-W_u^{\alpha,\beta}\right)^rf\|_{ L^{p(\cdot)}(0,\pi)}}{u^{\gamma+1}}  du  <\infty.$$
Moreover, the operator $I_{\varepsilon}^{\gamma,r}$ is bounded from $L^{p(\cdot)}(0,\pi)$ into itself (Proposition~\ref{prop:Iepsgamr}).
We consider the domain of ${\mathcal L}_{\alpha,\beta}^\gamma$
$$D_{p(\cdot)}({\mathcal L}_{\alpha,\beta}^\gamma)
    = \Big\{ f\in  L^{p(\cdot)}(0,\pi) \ : \  \lim_{\varepsilon\rightarrow 0^+} I_{\varepsilon}^{\gamma,r} f \text{ exists in } L^{p(\cdot)}(0,\pi) \Big\},$$
and we define
\begin{equation}\label{eqH2}
    {\mathcal L}_{\alpha,\beta}^\gamma f
        = \lim_{\varepsilon\rightarrow 0^+} I_{\varepsilon}^{\gamma,r} f, \quad f\in D_{p(\cdot)}({\mathcal L}_{\alpha,\beta}^\gamma).
\end{equation}
As it will be shown in Section~\ref{sec:7}, in the definition of ${\mathcal L}_{\alpha,\beta}^\gamma$ we can take any $r\in\Bbb N,\ r> \gamma$.
Next, we characterize the Jacobi potential space $H_{\alpha,\beta}^{\gamma,p(\cdot)}(0,\pi)$ as the domain of ${\mathcal L}_{\alpha,\beta}^\gamma$.
\begin{Th}\label{ThZ}
    Let $\gamma>0$ and $\alpha, \beta \geq -1/2$ such that $\alpha + \beta \neq -1$. Assume that $p\in{\mathcal B}(0,\pi)$.
    Then, $H_{\alpha,\beta}^{\gamma,p(\cdot)}(0,\pi)= D_{p(\cdot)}({\mathcal L}_{\alpha,\beta}^\gamma)$.
    Moreover, for every $f\in  D_{p(\cdot)}({\mathcal L}_{\alpha,\beta}^\gamma)$,
    $${\mathcal L}_{\alpha,\beta}^{-\gamma}{\mathcal L}_{\alpha,\beta}^{\gamma}f=f,$$
    and, for every $f\in  L^{p(\cdot)}(0,\pi)$,
    $${\mathcal L}_{\alpha,\beta}^{\gamma}{\mathcal L}_{\alpha,\beta}^{-\gamma}f=f.$$
\end{Th}

Segovia and Wheeden \cite{SW} characterized potential spaces by using Littlewood-Paley square functions.
In order to do this they introduced square functions involving fractional derivatives of the classical Poisson semigroup.
Inspired by \cite{SW}, Betancor, Fari\~na, Rodr\'{\i}guez-Mesa, Testoni and Torrea obtained characterizations using
vertical and area Littlewood-Paley functions for the potential spaces associated with the Hermite and
Ornstein-Uhlenbeck operators (\cite{BFRTT2}) and Schr\"odinger operators (\cite{BFRTT3}). We will characterize
our variable exponent Jacobi potential spaces by using Littlewood-Paley function defined via derivatives of the
Jacobi-Poisson semigroup.

Let $\gamma>0$ and $k \in \N$ such that $0<\gamma<k$. We consider the following Littlewood-Paley function
$$g_{\alpha,\beta}^{\gamma,k}(f)(\theta)
    = \left( \int_0^\infty \Big| t^{k-\gamma} \partial_t^k P_t^{\alpha, \beta} (f) (\theta) \Big|^2 \frac{dt}{t} \right)^{1/2}, \quad \theta \in (0,\pi).$$
We say that a measurable function $f \in L^{p(\cdot)}(0,\pi)$ is in $T_{\alpha, \beta}^{\gamma,k,p(\cdot)}(0,\pi)$ when
$g_{\alpha,\beta}^{\gamma,k}(f) \in L^{p(\cdot)}(0,\pi)$. On $T_{\alpha, \beta}^{\gamma,k,p(\cdot)}(0,\pi)$ we define the norm
$$\|f\|_{T_{\alpha, \beta}^{\gamma,k,p(\cdot)}(0,\pi)}
    =\|f\|_{L^{p(\cdot)}(0,\pi)}+  \|g_{\alpha,\beta}^{\gamma,k}(f)\|_{L^{p(\cdot)}(0,\pi)}, \quad f \in T_{\alpha, \beta}^{\gamma,k,p(\cdot)}(0,\pi).$$
Thus, $T_{\alpha, \beta}^{\gamma,k,p(\cdot)}(0,\pi)$ is a Banach space.

The space $T_{\alpha, \beta}^{\gamma,k,p(\cdot)}(0,\pi)$, which can be seen as a variable exponent Triebel-Lizorkin type space,
coincides with the variable exponent potential space $H_{\alpha,\beta}^{\gamma/2,p(\cdot)}(0,\pi)$.

\begin{Th}\label{Th2}
    Let $\alpha, \beta \geq -1/2$  such that $\alpha + \beta \neq -1$ and $0<\gamma<k$, $k \in \N$.
    Assume that $p \in \mathcal{B}(0,\pi)$.
    Then, $H_{\alpha,\beta}^{\gamma/2,p(\cdot)}(0,\pi) = T_{\alpha, \beta}^{\gamma,k,p(\cdot)}(0,\pi)$. Moreover, the norms
    $\| \cdot \|_{H_{\alpha,\beta}^{\gamma/2,p(\cdot)}(0,\pi)}$ and $\| \cdot \|_{T_{\alpha, \beta}^{\gamma,k,p(\cdot)}(0,\pi)}$
    are equivalent.
\end{Th}

Note that from Theorem~\ref{Th2} we deduce that the space $T_{\alpha, \beta}^{\gamma,k,p(\cdot)}(0,\pi)$ does not depend
on $k \in \N$ provided that $0<\gamma<k$. The result in Theorem~\ref{Th2} is new even when $p \in \mathcal{P}(0,\pi)$
is constant and it gives a new characterization of the Jacobi Sobolev spaces introduced in \cite{La}.

In order to prove Theorem~\ref{Th2} we need to show that certain square function
related to $g_{\alpha,\beta}^{\gamma,k}$, which involves fractional derivatives, is bounded on $L^{p(\cdot)}(0,\pi)$
. In \cite{SW} fractional derivatives were introduced. Suppose that $\gamma>0$ and
$F$ is a nice enough function defined in $(0,\pi) \times (0,\infty)$. The $\gamma$-th derivative $\partial_t^\gamma F$
is defined by
$$\partial_t^\gamma F(\theta,t)
    = \frac{e^{-i(m-\gamma)\pi}}{\Gamma(m-\gamma)} \int_0^\infty \partial_t^m F(\theta,t+s) s^{m-\gamma-1} ds,
    \quad \theta \in (0,\pi), \ t> 0,$$
where $m \in \N$ is such that $m-1 \leq \gamma < m$.

We consider the Littlewood-Paley function $g_{\alpha, \beta}^\gamma$ given by
$$g_{\alpha,\beta}^{\gamma}(f)(\theta)
    = \left( \int_0^\infty \Big| t^{\gamma} \partial_t^\gamma P_t^{\alpha, \beta} (f) (\theta) \Big|^2 \frac{dt}{t} \right)^{1/2}, \quad \theta \in (0,\pi).$$
The key relation between $g_{\alpha,\beta}^{\gamma,k}$ and $g_{\alpha,\beta}^{\gamma}$, $0<\gamma<k$, which allows to connect the spaces
$H_{\alpha,\beta}^{\gamma/2,p(\cdot)}(0,\pi)$ and $T_{\alpha, \beta}^{\gamma,k,p(\cdot)}(0,\pi)$, is the following
$$g_{\alpha,\beta}^{k-\gamma}(f)
    = g_{\alpha,\beta}^{\gamma,k}(\mathcal{L}_{\alpha,\beta}^{-\gamma/2}f), \quad f \in S_{\alpha, \beta}.$$

In \cite{KPX} Kyriazis, Petrushev and Xu defined Besov and Triebel-Lizorkin spaces associated with Jacobi expansions
with respect to $\big((-1,1),(1-x)^\alpha (1+x)^\beta dx\big)$. We now adapt the Triebel-Lizorkin definitions given in \cite{KPX}
to our Jacobi expansions in $\big( (0,\pi),d\theta \big)$. We take a function $\mathfrak{a} \in C_c^\infty(0,\infty)$ such that
$\supp \mathfrak{a} \subseteq [1/2,2]$ and $\inf_{t \in [3/5,5/3]} |\mathfrak{a}(t)|>0$. The following construction is independent of the election of  $\mathfrak{a}$ and, as it is said in  \cite{KPX}, we can add the condition that  $\mathfrak{a}(t)+ \mathfrak{a}(2t)=1$ for $t\in[1/2,1]$.  We define the sequence $\{\Phi_j^{\alpha,\beta}\}_{j \in \N}$ of functions
on $(0,\pi)^2$ as follows,
$$\Phi_0^{\alpha,\beta}(\theta,\varphi)
    = \phi_0^{\alpha,\beta}(\theta) \phi_0^{\alpha,\beta}(\varphi), \quad \theta, \varphi \in (0,\pi),$$
and, for every $j \in \N$, $j \geq 1$,
$$\Phi_j^{\alpha,\beta}(\theta,\varphi)
    = \sum_{n=0}^\infty \mathfrak{a} \Big( \frac{\lambda_n^{\alpha,\beta}}{2^{j-1}} \Big) \phi_n^{\alpha,\beta}(\theta) \phi_n^{\alpha,\beta}(\varphi), \quad \theta, \varphi \in (0,\pi).$$
If $\gamma \in \R$ and $0<p,q<\infty$, a function $f \in L^1(0,\pi)$ is in the Jacobi-Triebel-Lizorkin space
$F_{\alpha,\beta}^{\gamma,q,p}(0,\pi)$ provided that
$$\|f\|_{F_{\alpha,\beta}^{\gamma,q,p}(0,\pi)}
    = \Big\| \Big( \sum_{j=0}^\infty \big( 2^{j\gamma} \big| \Phi_j^{\alpha,\beta}(f)(\cdot) \big| \big)^q \Big)^{1/q} \Big\|_{L^p(0,\pi)}
    < \infty.$$
Here, for every $j \in \N$,
$$\Phi_j^{\alpha,\beta}(f)(\theta)
    = \int_0^\pi \Phi_j^{\alpha,\beta}(\theta,\varphi) f(\varphi) d \varphi, \quad \theta \in (0,\pi).$$
It would be interesting to investigate Jacobi-Triebel-Lizorkin spaces with variable exponent in the
$\big((-1,1),(1-x)^\alpha (1+x)^\beta dx\big)$ and $\big( (0,\pi),d\theta \big)$ settings.
This question will be considered on its whole generality in a forthcoming paper. Here we only introduce Jacobi-Triebel-Lizorkin
spaces with $\gamma>0$, $q=2$ and variable exponent $p(\cdot)$. Assume that $p \in \mathcal{P}(0,\pi)$.
A function $f \in L^{p(\cdot)}(0,\pi)$ is in $F_{\alpha,\beta}^{\gamma,2,p(\cdot)}(0,\pi)$ when
$$\|f\|_{F_{\alpha,\beta}^{\gamma,2,p(\cdot)}(0,\pi)}
    = \Big\| \Big( \sum_{j=0}^\infty \big( 2^{j\gamma} \big| \Phi_j^{\alpha,\beta}(f)(\cdot) \big| \big)^2 \Big)^{1/2} \Big\|_{L^{p(\cdot)}(0,\pi)}
    < \infty.$$
In the following theorem we identify the variable exponent Jacobi-Triebel-Lizorkin space $F_{\alpha,\beta}^{\gamma,2,p(\cdot)}(0,\pi)$
with the potential space $H_{\alpha,\beta}^{\gamma,p(\cdot)}(0,\pi)$.

\begin{Th}\label{Th3}
    Let $\alpha,\beta \geq -1/2$ and $\gamma>0$. Assume that $p \in \mathcal{B}(0,\pi)$. Then,
    $H_{\alpha,\beta}^{\gamma,p(\cdot)}(0,\pi) = F_{\alpha,\beta}^{\gamma,2,p(\cdot)}(0,\pi) $. Moreover, the norms
    $\| \cdot \|_{H_{\alpha,\beta}^{\gamma,p(\cdot)}(0,\pi)}$ and $\| \cdot \|_{F_{\alpha,\beta}^{\gamma,2,p(\cdot)}(0,\pi)}$ are equivalent.
\end{Th}

Note that as a special case of Theorem~\ref{Th3} we establish that the Jacobi potential space $H_{\alpha,\beta}^{\gamma,p}(0,\pi)$
considered by Langowski (\cite{La}) coincides with the Jacobi-Triebel-Lizorkin space
$F_{\alpha,\beta}^{\gamma,2,p}(0,\pi)$, for every $1<p<\infty$.

The paper is organized as follows. In Sections \ref{sec:2}, \ref{sec:3} and \ref{sec:4} we prove  that assertions (a), (b), (c) and (d) are true. Theorems \ref{ThZ}, \ref{Th2} and \ref{Th3}
are proved in Sections \ref{sec:7}, \ref{sec:5} and \ref{sec:6}, respectively.

Throughout this paper by $C$ and $c$ we always denote positive constants that can change in each occurrence.

\section{Dense subspaces} \label{sec:2}
This section deals with the proof of the  $W_{\alpha,\beta}^{k,p(\cdot)}$-density of  $S_{\alpha ,\beta}$ claimed in assertion (a) of Section \ref{sec:intro}.

Assume that $p\in{\mathcal P}(0,\pi)$.
According to \cite[Theorem 3.4.6]{DHHR} the space $L^{p'(\cdot)}(0,\pi)$ is isomorphic to
the dual space $(L^{p(\cdot)}(0,\pi))^*$ of $L^{p(\cdot)}(0,\pi)$. On the other hand, for every $k\in\Bbb N$,
${\phi}_k^{\alpha ,\beta}\in L^{\infty}(0,\pi)$. Then, ${\phi}_k^{\alpha ,\beta}\in L^{p'(\cdot)}(0,\pi)$,
$k\in\Bbb N$ (\cite[Theorem 3.3.11]{DHHR}). We define, for every $f\in L^{p(\cdot)}(0,\pi)$ and $k\in\Bbb N$,
$$c_k^{\alpha ,\beta}(f)=\int_0^{\pi}{{\phi}_k^{\alpha ,\beta}(\theta) f(\theta) d\theta}.$$
By \cite[Theorem 3.4.12]{DHHR} the space $C_c^{\infty}(0,\pi)$ is dense in $L^{p(\cdot)}(0,\pi)$.

\begin{Prop}\label{propS1}
    Let $\alpha,\beta\geq-1/2$ and $p\in{\mathcal P}(0,\pi)$.
    The space $S_{\alpha ,\beta}=\mathrm{span}\{{\phi}_k^{\alpha ,\beta}\}_{k\in\Bbb N}$ is dense in $L^{p(\cdot)}(0,\pi)$.
\end{Prop}

\begin{proof}
    Since $C_c^{\infty}(0,\pi)$ is a dense subspace of $L^{p(\cdot)}(0,\pi)$, it is sufficient to see that
    $C_c^{\infty}(0,\pi)$ is contained in the closure of $S_{\alpha ,\beta}$ in $L^{p(\cdot)}(0,\pi)$.
    Let $g\in C_c^{\infty}(0,\pi)$. By using integration by parts we deduce that, for every $m\in\Bbb N$,
    there exists $C_m>0$ such that  $|c_k^{\alpha ,\beta}(g)|<C_m(k+1)^{-m}$, $k\in\Bbb N$. Hence,
    $$S_n^{\alpha ,\beta}(g)
        =\sum_{k=0}^n{c_k^{\alpha ,\beta}(g){\phi}_k^{\alpha ,\beta}}\ \longrightarrow\ g\ ,\ \ \mbox{as}\ n\rightarrow\infty,
        \quad \text{in } L^{\infty}(0,\pi). $$
    Hence, according to \cite[Theorem 3.3.11]{DHHR},
    $S_n^{\alpha ,\beta}(\phi)\rightarrow\phi$, as $n\rightarrow\infty$, in $L^{p(\cdot)}(0,\pi)$.
\end{proof}

\begin{Cor}\label{corS2}
    Let $\alpha,\beta\geq -1/2$ and $p\in{\mathcal P}(0,\pi)$.
    If $f\in L^{p(\cdot)}(0,\pi)$ and $c_k^{\alpha ,\beta}(f)=0,\ k\in\Bbb N$, then $f=0$.
\end{Cor}

\begin{proof}
    Since $p\in{\mathcal P}(0,\pi)$, $p'$ is also in ${\mathcal P}(0,\pi)$.
    Then, by Proposition \ref{propS1}, $S_{\alpha ,\beta}$ is dense in $L^{p'(\cdot)}(0,\pi)$.
    Assume that $f\in L^{p(\cdot)}(0,\pi)$ is such that $c_k^{\alpha ,\beta}(f)=0,\ k\in\Bbb N$.
    The norm conjugate formula (\cite[Corollary 3.2.14]{DHHR}) leads to
    $$\int_0^{\pi}{f(\theta)g(\theta)d\theta}
        =0,$$
    for every $g\in L^{p'(\cdot)}(0,\pi)$. By using again the norm conjugate formula (duality) we conclude that $f=0$.
\end{proof}

We can improve the result in Proposition~\ref{propS1} when the function $p(\cdot)$ satisfies additional conditions.
According to \cite[Theorem 1]{Mu6}, if $1<p<\infty$ and $f \in L^p(0,\pi)$, then
$$f
    = \lim_{n \to \infty} \sum_{k=0}^n c_k^{\alpha,\beta}(f) \phi_k^{\alpha,\beta},$$
where the convergence is understood in $L^p(0,\pi)$. We now establish this property in
$L^p_w(0,\pi)$, $1<p<\infty$ and $w \in A_p(0,\pi)$, and in $L^{p(\cdot)}(0,\pi)$
when the function $p(\cdot)$ is as in \cite[Theorem 1.3]{CFMP}.

\begin{Prop}\label{Prop:B1}
    Let $\alpha,\beta \geq -1/2$.
    \begin{itemize}
        \item[$(i)$] If $1<p<\infty$ and $w \in A_p(0,\pi)$, there exists $C>0$ such that, for every $n \in \N$,
        $$\Big\| \sum_{k=0}^n c_k^{\alpha,\beta}(f) \phi_k^{\alpha,\beta} \Big\|_{L^p_w(0,\pi)}
            \leq C \|f\|_{L^p_w(0,\pi)}, \quad f \in L^p_w(0,\pi),$$
        and
        $$\lim_{n \to \infty} \sum_{k=0}^n c_k^{\alpha,\beta}(f) \phi_k^{\alpha,\beta}
            = f, \quad f \in L^p_w(0,\pi),$$
        in the sense of convergence in $L^p_w(0,\pi)$.

        \item[$(ii)$] Assume that $p \in \mathcal{B}(0,\pi)$. Then, there exists $C>0$ such that, for every $n \in \N$,
        $$\Big\| \sum_{k=0}^n c_k^{\alpha,\beta}(f) \phi_k^{\alpha,\beta} \Big\|_{L^{p(\cdot)}(0,\pi)}
            \leq C \|f\|_{L^{p(\cdot)}(0,\pi)}, \quad f \in L^{p(\cdot)}(0,\pi),$$
        and
        $$\lim_{n \to \infty} \sum_{k=0}^n c_k^{\alpha,\beta}(f) \phi_k^{\alpha,\beta}
            = f, \quad f \in L^{p(\cdot)}(0,\pi),$$
        in the sense of convergence in $L^{p(\cdot)}(0,\pi)$.
    \end{itemize}
\end{Prop}

\begin{proof}[Proof of Proposition~\ref{Prop:B1}, $(i)$]
    In order to prove this property we proceed as in the proof of \cite[Theorem 2]{Ke}.
    Let $1<p<\infty$ and $w \in A_p(0,\pi)$. Suppose that $f \in L^p_w(0,\pi)$ and $n \in \N$.
    We define
    $$S_n f(\theta)
        = \sum_{k=0}^n c_k^{\alpha,\beta}(f) \phi_k^{\alpha,\beta}(\theta), \quad \theta \in (0,\pi).$$
    As in \cite[p. 13]{Ke} we have that
    \begin{equation}\label{eq:A1}
        |S_n f(\theta)|
            \leq C \sum_{\ell=1}^3 J_\ell^{\alpha,\beta,n} f (\theta), \quad \theta \in (0,\pi),
    \end{equation}
    where the operators $J_\ell^{\alpha,\beta,n}$, $\ell=1,2,3$ can be estimated as follows. Firstly, for $J_1^{\alpha,\beta,n}$ we get
    \begin{align*}
        J_1^{\alpha,\beta,n} f (\theta)
            \leq & C \frac{\big(\sin \frac{\theta}{2} \big)^{\alpha+1/2} \big(\cos \frac{\theta}{2} \big)^{\beta+1/2}}{\big(\sin \frac{\theta}{2}  + \frac{1}{n+1}\big)^{\alpha+1/2} \big(\cos \frac{\theta}{2} + \frac{1}{n+1}\big)^{\beta+1/2}} \\
                 & \times \int_0^\pi \frac{\big(\sin \frac{\varphi}{2} \big)^{\alpha+1/2} \big(\cos \frac{\varphi}{2} \big)^{\beta+1/2}}{\big(\sin \frac{\varphi}{2}  + \frac{1}{n+1}\big)^{\alpha+1/2} \big(\cos \frac{\varphi}{2} + \frac{1}{n+1}\big)^{\beta+1/2}} |f(\varphi)|  d \varphi \\
            \leq & C \int_0^\pi |f(\varphi)|  d \varphi, \quad \theta \in (0,\pi).
    \end{align*}
    Then, H\"older's inequality implies that
    \begin{equation}\label{eq:A2}
        \int_0^\pi |J_1^{\alpha,\beta,n} f (\theta)|^p w(\theta) d\theta
            \leq C \int_0^\pi |f(\theta)|^p w(\theta) d\theta,
    \end{equation}
    because $L^p_w(0,\pi) \subseteq L^1(0,\pi)$.

    For $J_2^{\alpha,\beta,n}$ the following estimate holds
    \begin{align*}
        J_2^{\alpha,\beta,n} f (\theta)
            & \leq C \frac{\big(\sin \frac{\theta}{2} \big)^{\alpha+1/2} \big(\cos \frac{\theta}{2} \big)^{\beta+1/2}}{\big(\sin \frac{\theta}{2}  + \frac{1}{n+1}\big)^{\alpha+1/2} \big(\cos \frac{\theta}{2} + \frac{1}{n+1}\big)^{\beta+1/2}} \\
            & \qquad \times \Big| \int_0^\pi \frac{\sin \varphi}{\sin \frac{\theta + \varphi}{2} \sin \frac{\theta - \varphi}{2}}
                                             \frac{\big(\sin \frac{\varphi}{2} \big)^{\alpha+3/2} \big(\cos \frac{\varphi}{2} \big)^{\beta+3/2}}{\big(\sin \frac{\varphi}{2}  + \frac{1}{n}\big)^{\alpha+3/2} \big(\cos \frac{\varphi}{2} + \frac{1}{n}\big)^{\beta+3/2}}
                                             b_n(\varphi) f(\varphi)  d \varphi \Big|,
    \end{align*}
    where $\sup_{k \in \N} |b_k(\varphi)| \leq C$, $\varphi \in (0,\pi)$. We can write (see \cite[p. 14]{Ke})
    $$\frac{\sin \varphi}{\sin \frac{\theta + \varphi}{2} \sin \frac{\theta - \varphi}{2}}
        = \frac{1}{\sin \frac{\theta - \varphi}{2}} + R(\theta,\varphi), \quad \theta, \varphi \in (0,\pi), \ \theta  \neq \varphi,$$
    being
    $$|R(\theta,\varphi)|
        \leq C \left\{ \begin{array}{ll}
                            \dfrac{1}{\sin \frac{\theta}{2} + \sin \frac{\varphi}{2}}, & 0<\theta<\pi/2 \\
                            & \\
                            \dfrac{1}{\cos \frac{\theta}{2} + \cos \frac{\varphi}{2}}, & \pi/2<\theta<\pi, \\
                       \end{array} \right. \quad \varphi \in (0,\pi).$$
    Thus, by defining
    $$g(\varphi)
        = \frac{\big(\sin \frac{\varphi}{2} \big)^{\alpha+3/2} \big(\cos \frac{\varphi}{2} \big)^{\beta+3/2}}{\big(\sin \frac{\varphi}{2}  + \frac{1}{n}\big)^{\alpha+3/2} \big(\cos \frac{\varphi}{2} + \frac{1}{n}\big)^{\beta+3/2}}
          b_n(\varphi) f(\varphi), \quad \varphi \in (0,\pi),$$
    we obtain
    \begin{equation}\label{eq:A3}
        J_2^{\alpha,\beta,n} f (\theta)
            \leq C \Big[ |(Hg)(\theta)| + S^1(|g|)(\theta) + S^2(|g|)(\theta)\Big], \quad \theta \in (0,\pi),
    \end{equation}
    where
    $$(Hg)(\theta)
        = \text{P.V.} \int_0^\pi \frac{g(\varphi)}{\sin \frac{\theta-\varphi}{2}} d\varphi, \quad \text{a.e. } \theta \in (0,\pi),$$
    $$(S^1g)(\theta)
        = \int_0^\pi \frac{g(\varphi)}{\sin \frac{\theta}{2} + \sin \frac{\varphi}{2}} d\varphi, \quad \theta \in (0,\pi),$$
    and
    $$(S^2g)(\theta)
        = \int_0^\pi \frac{g(\varphi)}{\cos \frac{\theta}{2} + \cos \frac{\varphi}{2}} d\varphi, \quad \theta \in (0,\pi).$$
    The operator $H$ is a singular integral operator related to the Hilbert transform and $S^j$, $j=1,2$, are Stieltjes type operators.
    It is well-known (\cite{HMW}) that $H$ is bounded from $L^p_w(0,\pi)$ into itself. In \cite[Lemma 6]{Ke} it was
    established that $S^1$ and $S^2$ are bounded from $L^p_w(0,\pi)$ into itself. Then, \eqref{eq:A3} implies that
    \begin{equation}\label{eq:A4}
        \int_0^\pi |J_2^{\alpha,\beta,n} f (\theta)|^p w(\theta) d\theta
            \leq C \int_0^\pi |g(\theta)|^p w(\theta) d\theta
            \leq C \int_0^\pi |f(\theta)|^p w(\theta) d\theta.
    \end{equation}
    In a similar way we can see
    \begin{equation}\label{eq:A5}
        \int_0^\pi |J_3^{\alpha,\beta,n} f (\theta)|^p w(\theta) d\theta
            \leq C \int_0^\pi |f(\theta)|^p w(\theta) d\theta.
    \end{equation}
    By putting together \eqref{eq:A1}, \eqref{eq:A2}, \eqref{eq:A4} and \eqref{eq:A5} we conclude that
    \begin{equation*}\label{eq:A6}
        \|S_n f\|_{L^p_w(0,\pi)}
            \leq C \| f \|_{L^p_w(0,\pi)}.
    \end{equation*}
    Note that the constant $C>0$ does not depend on $n \in \N$ and $f \in L^p_w(0,\pi)$.

    Since $C_c^\infty(0,\pi)$ is a dense subspace of $L^p_w(0,\pi)$ and for every $h \in C_c^\infty(0,\pi)$,
    $$\lim_{n \to \infty} S_n h = h,  \quad \text{uniformly in} (0,\pi),$$
    and hence in $L^p_w(0,\pi)$; standard arguments allow
    us to show that, for every $f \in L^p_w(0,\pi)$,
    $$\lim_{n \to \infty} S_n f
        = f, \quad \text{in } L^p_w(0,\pi).$$
\end{proof}

\begin{proof}[Proof of Proposition~\ref{Prop:B1}, $(ii)$]
    From the property established in Proposition~\ref{Prop:B1}, $(i)$, and according to \cite[Theorem 1.3]{CFMP}
    we deduce that there exists $C>0$ such that, for every $n \in \N$,
    \begin{equation}\label{eq:A8}
        \|S_n f\|_{L^{p(\cdot)}(0,\pi)}
            \leq C \| f \|_{L^{p(\cdot)}(0,\pi)}, \quad f \in L^{p(\cdot)}(0,\pi).
    \end{equation}
    By \cite[Theorem 3.3.1]{DHHR}, $C_c^\infty(0,\pi) \subseteq L^{p_+}(0,\pi) \subseteq L^{p(\cdot)}(0,\pi)$ and the
    inclusions are continuous. Hence, for every $h \in C_c^\infty(0,\pi)$,
    $$\lim_{n \to \infty} S_n(h) = h, \quad \text{in } L^{p(\cdot)}(0,\pi).$$
    Since $C_c^\infty(0,\pi)$
    is dense in $L^{p(\cdot)}(0,\pi)$ we deduce from \eqref{eq:A8} that, for every $f \in L^{p(\cdot)}(0,\pi)$,
    $$\lim_{n \to \infty} S_n f = f, \quad \text{in } L^{p(\cdot)}(0,\pi).$$
\end{proof}

We are going to see that $S_{\alpha,\beta}$ is a dense subspace of $W^{k,p(\cdot)}_{\alpha ,\beta}(0,\pi)$.

\begin{Prop}\label{propS3}
    Let $\alpha,\beta\geq -1/2$, $k\in\Bbb N$ and $p\in{\mathcal B}(0,\pi)$.
    Then, $S_{\alpha ,\beta}$ is a dense subspace of $W^{k,p(\cdot)}_{\alpha ,\beta}(0,\pi)$.
\end{Prop}

\begin{proof}
    We proceed following the ideas in the proof of \cite[Proposition 2]{BFRTT1} (see also \cite[Proposition 3.2]{La}).
    Note firstly that, since $L^{p(\cdot)}(0,\pi)\subseteq L^{p_-}(0,\pi)$ (\cite[Theorem 3.3.1]{DHHR}),
    $W^{k,p(\cdot)}_{\alpha ,\beta}(0,\pi)\subseteq W^{k,p_-}_{\alpha ,\beta}(0,\pi)$, where the last Sobolev
    type space $W^{k,p_-}_{\alpha ,\beta}(0,\pi)$ (with constant exponent $p_-$) was studied by Langowski \cite{La}.

    Let $f\in W^{k,p(\cdot)}_{\alpha ,\beta}(0,\pi)$.
    The maximal operator $W_*^{\alpha ,\beta}$ associated with $\{W_t^{\alpha ,\beta}\}_{t>0}$ is defined by
    $$W_*^{\alpha ,\beta}(f)=\sup_{t>0}|W_t^{\alpha ,\beta}(f)|.$$
    According to \cite[Theorem A, and (3)]{NoSj3} we have that
    \begin{equation}\label{eqS2}
        |W_t^{\alpha ,\beta}(\theta,\varphi)|
             \leq C {e^{-c(\theta-\varphi)^2/t} \over{\sqrt{t}}}, \quad \theta,\varphi\in(0,\pi)\ \mbox{and}\ t>0.
    \end{equation}
    From (\ref{eqS2}) we deduce that
    $$W_t^{\alpha ,\beta}(f)
        \leq C{\mathcal M}_c(f),$$
    where ${\mathcal M}_c$ denotes the centered Hardy-Littlewood maximal operator.
    Then, by \cite[Theorem 4.3.8]{DHHR} $W_*^{\alpha ,\beta}$ is a bounded  (sublinear) operator from $L^{p(\cdot)}(0,\pi)$ into itself.
    It is clear that, for every $\phi\in S_{\alpha,\beta}$,
    $$\lim_{t \to 0^+}W_t^{\alpha ,\beta}(\phi) = \phi, \quad  \text{in } L^{p(\cdot)}(0,\pi).$$
   Then, since $S_{\alpha,\beta}$ is dense in $L^{p(\cdot)}(0,\pi)$ (Proposition \ref{propS1}), we obtain that,
    $$\lim_{t \to 0^+} W_t^{\alpha ,\beta}(f) = f, \quad \text{in } L^{p(\cdot)}(0,\pi).$$

    By \cite[Lemmas 3.1 and 3.3]{La}
    \begin{equation}\label{eqS1}
        c_m^{\alpha +\ell,\beta +\ell}\left({\mathbb{D}}_{\alpha,\beta}^\ell f\right)
            =(-1)^\ell \sqrt{(m+1)_\ell(m+\ell+\alpha+\beta +1)_\ell} \ c_{m+\ell}^{\alpha ,\beta}(f), \quad \ell, m \in\Bbb N,\ \ 0\leq \ell\leq k.
    \end{equation}
    Here and in the sequel we denote by $(z)_\ell$, $z>0$, the $\ell$-Pochhammer symbol, that is,
    \begin{equation}\label{eq:9.1}
        (z)_\ell=z(z+1)  \cdots (z+\ell-1), \quad \ell\in\Bbb N,\ \ \ell\geq 1 \quad \mbox{and} \quad (z)_0=1.
    \end{equation}

    By taking into account \cite[(1)]{La} we can differentiate term by term inside the series and
    \cite[Lemma 3.1]{La} and (\ref{eqS1}) lead to
    \begin{align*}
        {\mathbb{D}}_{\alpha,\beta}^\ell W_t^{\alpha ,\beta}f
        & =\sum_{m=0}^{\infty}{e^{-t{\lambda}_m^{\alpha ,\beta}}c_{m}^{\alpha ,\beta}(f) \ {\mathbb{D}}_{\alpha,\beta}^\ell {\phi}_m^{\alpha ,\beta}} \\
        & = \sum_{m=\ell}^{\infty}{e^{-t{\lambda}_m^{\alpha ,\beta}}(-1)^m \sqrt{(m-\ell+1)_\ell(m+\alpha+\beta+1)_\ell}
                \ c_{m}^{\alpha ,\beta}(f) \ {\phi}_{m-\ell}^{\alpha+\ell ,\beta+\ell}} \\
        & = \sum_{m=\ell}^{\infty}{e^{-t{\lambda}_m^{\alpha ,\beta}}c_{m-\ell}^{\alpha+\ell ,\beta+\ell}\left({\mathbb{D}}_{\alpha,\beta}^\ell f\right){\phi}_{m-\ell}^{\alpha+\ell ,\beta+\ell}} \\
        & = \sum_{m=0}^{\infty}{e^{-t{\lambda}_m^{\alpha+\ell ,\beta+\ell}}c_{m}^{\alpha+\ell ,\beta+\ell}\left({\mathbb{D}}_{\alpha,\beta}^\ell f\right){\phi}_{m}^{\alpha+\ell ,\beta+\ell}}, \quad  \ell\in\Bbb N,\  0\leq \ell\leq k.
    \end{align*}
    Hence, for every $\ell\in\Bbb N$, $0\leq \ell\leq k$,
    $$\lim_{t \to 0^+} {\mathbb{D}}_{\alpha,\beta}^\ell W_t^{\alpha ,\beta}f = {\mathbb{D}}_{\alpha,\beta}^\ell f, \quad \text{in } L^{p(\cdot)}(0,\pi).$$

    Let $\varepsilon >0$. There exists $t_0>0$ such that, for every $0<t<t_0$,
    $$\|{\mathbb{D}}_{\alpha,\beta}^\ell W_t^{\alpha ,\beta}f-{\mathbb{D}}_{\alpha,\beta}^\ell f\|_{L^{p(\cdot)}(0,\pi)}
        <\varepsilon, \quad \ell\in\Bbb N,\ \ 0\leq l\leq k.$$
    On the other hand, by using \cite[(1)]{La}, \cite[Theorem 3.3.11]{DHHR} and H\"{o}lder inequality we get, for every
    $\theta\in (0,\pi)$ and $\ell,m \in\Bbb N$,
    $$\left|c_{m}^{\alpha+\ell ,\beta+\ell}\left({\mathbb{D}}_{\alpha,\beta}^\ell f\right)\right|\left|{\phi}_{m}^{\alpha+\ell ,\beta+\ell}(\theta)\right|
        \leq C\|{\mathbb{D}}_{\alpha,\beta}^\ell f\|_{L^{p_-}(0,\pi)}(m+1)^{\alpha+\beta+2\ell+2}.$$
    Hence, there exists $m_0\in\Bbb N$, $m_0\geq k$, such that
    \begin{align*}
        &  \left\|\sum_{m=M+1}^{\infty}{e^{-t_0{\lambda}_m^{\alpha+\ell ,\beta+\ell}}c_{m}^{\alpha+\ell ,\beta+\ell}\left({\mathbb{D}}_{\alpha,\beta}^\ell f\right){\phi}_{m}^{\alpha+\ell ,\beta+\ell}}\right\|_{L^{p(\cdot)}(0,\pi)} \\
        & \qquad \leq C\sum_{m=m_0+1}^{\infty}{e^{-t_0(m+{{\alpha+\beta+2\ell}\over 2})^2}(m+1)^{\alpha+\beta+2\ell +2}}<\varepsilon, \quad \ell\in\Bbb N,\ 0\leq \ell\leq k,\ M\in\Bbb N,\ M\geq m_0.
    \end{align*}
    Then,
    $$\left\|\sum_{m=0}^{m_0}{e^{-t_0{\lambda}_m^{\alpha ,\beta}}c_{m}^{\alpha ,\beta}(f){\phi}_{m}^{\alpha ,\beta}}-f\right\|_{W^{k,p(\cdot)}_{\alpha ,\beta}(0,\pi)}
        <2\varepsilon.$$
    Thus, we have proved that $f$ is in the closure of $S_{\alpha,\beta}$ in $W^{k,p(\cdot)}_{\alpha ,\beta}(0,\pi)$ and the proof is finished.
\end{proof}

\section{Jacobi multipliers in weighted $L^p$-spaces} \label{sec:3}
This section deals, among other things, with the proof of the $H_{\alpha,\beta}^{k/2,p(\cdot)}$-density of $S_{\alpha,\beta}$ claimed in  assertions (a) and (d) of Section \ref{sec:intro}.

Let $m=(m_k)_{k=0}^{\infty}$ be a bounded sequence of real numbers.
The Jacobi multiplier $T_m^{\alpha,\beta}$ associated with $m$ is defined by
$$T_m^{\alpha,\beta}f=\sum_{k=0}^{\infty}{m_kc_k^{\alpha,\beta}(f){\phi}_k^{\alpha,\beta}}, \quad f\in L^2(0,\pi).$$
Plancherel's equality implies that $T_m^{\alpha,\beta}$ is bounded on $L^2(0,\pi)$.
Sufficient conditions which allow to extend $T_m^{\alpha,\beta}$ as a bounded operator to $L^p(0,\pi)$
and to certain weighted $L^p(0,\pi)$ spaces have been established by several authors
(see \cite{BU}, \cite{BC}, \cite{CS}, \cite{GT}, \cite{M}, \cite{Mu5}, \cite{MS}  and \cite{Wr},   amongst others).

The goal of this section is to establish a multiplier theorem in  $L^{p(\cdot)}(0,\pi)$. Previously we need  to show a multiplier
result for  $L^p_w(0,\pi)$ when $w\in A_p(0,\pi)$. In order to achieve this we invoke a general multiplier theorem due to Meda \cite {Me1} (see also \cite{Wr}).

Let $-\infty<a<\left({{\alpha+\beta+1}\over 2}\right)^2$. We consider the operator
$${\mathcal L}_{\alpha,\beta;a}
    ={\mathcal L}_{\alpha,\beta}-a.$$
It is clear that, for every $k\in\Bbb N$, ${\phi}_k^{\alpha,\beta}$ is an eigenfunction for ${\mathcal L}_{\alpha,\beta;a}$ associated with the eigenvalue
$${\lambda}_k^{\alpha,\beta;a}
    =\left(k+{{\alpha+\beta+1}\over 2}\right)^2-a=k(k+\alpha+\beta+1)+\left({{\alpha+\beta+1}\over 2}\right)^2-a.$$
${\mathcal L}_{\alpha,\beta;a}$ is a nonnegative and selfadjoint operator on $L^2(0,\pi)$.
Moreover, ${\mathcal L}_{\alpha,\beta;a}$ generates a (heat) semigroup $\{W_t^{\alpha ,\beta;a}\}_{t>0}$ on $L^2(0,\pi)$, given by
$$W_t^{\alpha ,\beta;a}(f)
    =\int_0^{\pi}{W_t^{\alpha ,\beta;a}(\theta,\varphi)f(\varphi)d\varphi}, \quad f\in L^2(0,\pi), \quad t>0,$$
and
$$W_t^{\alpha ,\beta;a}(\theta,\varphi)
    =\sum_{k=0}^{\infty}{e^{-t{\lambda}_k^{\alpha ,\beta;a}}{\phi}_k^{\alpha ,\beta}(\theta){\phi}_k^{\alpha ,\beta}(\varphi)}, \quad \ \theta,\varphi\in(0,\pi),\ \mbox{and}\ t>0.$$
According to \cite[Theorem A, (3) and (9)]{NoSj3} we have that
$$\left|W_t^{\alpha ,\beta;a}(\theta,\varphi)\right|
    \leq Ce^{-(({{\alpha+\beta+1}\over 2})^2-a)t}\;{{e^{-c(\theta-\varphi)^2/t}}\over{\sqrt{t}}}, \quad \theta,\varphi\in(0,\pi)\ \mbox{and}\ t>0.$$

Let $\gamma\in\Bbb R\backslash\{0\}$. The imaginary power ${\mathcal L}_{\alpha,\beta;a}^{i\gamma}$
of ${\mathcal L}_{\alpha,\beta;a}$ is the spectral multiplier $g({\mathcal L}_{\alpha,\beta;a})$ where $g(x)=x^{i\gamma}$, $x>0$, that is,
$${\mathcal L}_{\alpha,\beta;a}^{i\gamma}(f)
    =\sum_{k=0}^{\infty}{({\lambda}_k^{\alpha ,\beta;a})^{i\gamma}c_k^{\alpha ,\beta}(f){\phi}_k^{\alpha ,\beta}}, \quad f\in L^2(0,\pi).$$

The operator ${\mathcal L}_{\alpha,\beta;a}^{i\gamma}$ can be seen as a Laplace transform type
multiplier for ${\mathcal L}_{\alpha,\beta;a}$. Then, a general result due to Stein \cite[Corollary 3, p. 121]{Ste1}
applies to deduce that ${\mathcal L}_{\alpha,\beta;a}^{i\gamma}$ can be extended from
$L^2(0,\pi)\cap L^p(0,\pi)$ to  $L^p(0,\pi)$ as a bounded operator on  $L^p(0,\pi)$, for every $1<p<\infty$. Also, by proceeding as in \cite{NoSj2} we can see that
${\mathcal L}_{\alpha,\beta;a}^{i\gamma}$ is a Calder\'on-Zygmund operator in the sense of a space of homogeneous type $((0,\pi),d\theta,|\cdot|)$,
where $|\cdot|$ stands for the Euclidean metric.
 Then, ${\mathcal L}_{\alpha,\beta;a}^{i\gamma}$
defines a bounded operator from $L^p_{w}(0,\pi)$ into itself, for every $1<p<\infty$ and $w\in A_p(0,\pi)$.
Moreover, classical arguments (see for instance, \cite[Chapter 7, Section 4]{Duo}) allow us to obtain that,
for every $1<p<\infty$ and $w\in A_p(0,\pi)$,
\begin{equation}\label{eqM1}
    \| {\mathcal L}_{\alpha,\beta;a}^{i\gamma}\|_{L^p_{w}(0,\pi)\rightarrow L^p_{w}(0,\pi)}
        \leq C_{p,w}e^{\pi |\gamma|/2},
\end{equation}
where $C_{p,w}>0$ does not depend on $\gamma$. Estimation (\ref{eqM1}) shows an exponential increase with respect to
$|\gamma|$ of the operator norm $\| {\mathcal L}_{\alpha,\beta;a}^{i\gamma}\|_{L^p_{w}(0,\pi)\rightarrow L^p_{w}(0,\pi)}$ which
is not sufficient to obtain our multiplier result. Actually, the exponential behavior in
(\ref{eqM1}) can be replaced by a polynomial growth. Indeed, according to \cite[Theorem 1.3 and Remarks 1.4 and 1.5]{CA}
we have that, for every $1<p<\infty$ and $w\in A_p(0,\pi)$,
\begin{equation*}\label{eqM1.5}
    \| {\mathcal L}_{\alpha,\beta;a}^{i\gamma}\|_{L^p_{w}(0,\pi)\rightarrow L^p_{w}(0,\pi)}
        \leq C_{p,w}(1+|\gamma|),
\end{equation*}
where $C_{p,w}>0$ does not depend on $\gamma$.

We now establish our result concerning the $ L^p_{w}(0,\pi)$-boundedness of spectral multipliers for the operator $ {\mathcal L}_{\alpha,\beta;a}.$

\begin{Prop}\label{propM1}
    Let $1<p<\infty$, $\alpha,\beta\geq -1/2$ and $-\infty<a<\left({{\alpha+\beta+1}\over 2}\right)^2$.
    Assume that:
    \begin{itemize}
        \item[$(i)$] $m$ is a bounded holomorphic function on $\{z\in\Bbb C\ : \ \mathrm{Re}\ z>0\}$; or
        \item[$(ii)$] $m \in C^\infty(0,\pi)$ and for every $\ell \in \N$
        \begin{equation}\label{eq:Mihlin}
            \sup_{x\in(0,\infty)}\left|x^\ell {{d^\ell }\over{dx^\ell }}m(x)\right| <\infty.
        \end{equation}
    \end{itemize}
    Then, the spectral multiplier $m({\mathcal L}_{\alpha,\beta;a})$ related to the operator ${\mathcal L}_{\alpha,\beta;a}$ given by
    \begin{equation}\label{eqM2}
        m({\mathcal L}_{\alpha,\beta;a}) f
            =\sum_{k=0}^{\infty}{m({\lambda}_k^{\alpha,\beta;a})c_k^{\alpha,\beta}(f){\phi}_k^{\alpha,\beta}},
    \end{equation}
    is bounded from  $L^p_{w}(0,\pi)$ into itself, for every $w\in A_p(0,\pi)$.
\end{Prop}

This result can be proved as in \cite[Theorem 3 or Corollary 1]{Me1}. By using now
\cite[Theorem 1.3]{CFMP} we deduce from Proposition \ref{propM1} the following $L^{p(\cdot)}$-boundedness result for
spectral multipliers associated with $ {\mathcal L}_{\alpha,\beta;a}$.

\begin{Prop}\label{propM2}
    Let  $\alpha,\beta\geq -1/2$ and  $-\infty<a<\left({{\alpha+\beta+1}\over 2}\right)^2$.
    Assume that $p \in{\mathcal B}(0,\pi)$.
    If $m$ satisfies condition $(i)$ or $(ii)$ of Proposition~\ref{propM1}, then the spectral multiplier $m({\mathcal L}_{\alpha,\beta;a})$
    given by (\ref{eqM2}) defines a bounded operator from $ L^{p(\cdot)}(0,\pi)$ into itself.
\end{Prop}

The negative powers of $ {\mathcal L}_{\alpha,\beta}$ defined in \eqref{eq:negpow} are spectral multipliers for the Jacobi operator that will be useful in the sequel.
Suppose that $\gamma>0$ and $\alpha + \beta \neq -1$.
Since ${\lambda}_k^{\alpha,\beta}\geq\left({{\alpha+\beta+1}\over 2}\right)^2$, $k\in\Bbb N$,
the operator ${\mathcal L}_{\alpha,\beta}^{ -\gamma}$ is bounded from $ L^2(0,\pi)$ into itself.

We take $a={1\over 2}\left({{\alpha+\beta+1}\over 2}\right)^2$. We can write
$${\mathcal L}_{\alpha,\beta}^{ -\gamma}f
    =\sum_{k=0}^{\infty}{({\lambda}_k^{\alpha,\beta;a}+a)^{-\gamma}c_k^{\alpha,\beta}(f){\phi}_k^{\alpha,\beta}}
    =T_{m_{\gamma}}^{\alpha,\beta;a}(f), \quad f\in L^2(0,\pi),$$
where $m_{\gamma}(z)=(z+a)^{-\gamma}$, $z\in\Bbb C$, $\mathrm{Re}\ z>0$. Since $m_{\gamma}$ is a bounded
holomorphic function on $\{z\in\Bbb C\ : \ \mathrm{Re}\ z>0\}$ from Propositions \ref{propM1} and \ref{propM2} we deduce the following.

\begin{Prop}\label{propM3}
    Let $\gamma>0$ and $\alpha,\beta\geq -1/2$ such that $\alpha + \beta \neq -1$.
    \begin{enumerate}
        \item[(a)] If $1<p<\infty$ and $w\in A_p(0,\pi)$, then ${\mathcal L}_{\alpha,\beta}^{ -\gamma}$
        can be extended from  $ L^2(0,\pi)\cap  L^p_{w}(0,\pi)$ to  $ L^p_{w}(0,\pi)$ as a bounded operator from $ L^p_{w}(0,\pi)$ into itself.
        \item[(b)] If $p \in{\mathcal B}(0,\pi)$,
        then ${\mathcal L}_{\alpha,\beta}^{ -\gamma}$ defines a bounded operator from $ L^{p(\cdot)}(0,\pi)$ into itself.
    \end{enumerate}
\end{Prop}

We also have the injectivity of ${\mathcal L}_{\alpha,\beta}^{ -\gamma}$ on $ L^p_{w}(0,\pi)$ and $ L^{p(\cdot)}(0,\pi)$.

\begin{Prop}\label{propM4}
    Let $\gamma>0$ and $\alpha,\beta\geq -1/2$ such that $\alpha + \beta \neq -1$.
    \begin{enumerate}
        \item[(a)] If $1<p<\infty$ and $w\in A_p(0,\pi)$, then ${\mathcal L}_{\alpha,\beta}^{ -\gamma}$
        is one to one on  $ L^p_{w}(0,\pi)$.
        \item[(b)] Assume that $p \in{\mathcal B}(0,\pi)$.
        Then, ${\mathcal L}_{\alpha,\beta}^{ -\gamma}$ is one to one on $ L^{p(\cdot)}(0,\pi)$.
    \end{enumerate}
\end{Prop}

\begin{proof}
    We prove (b). Property (a) can be shown in a similar way. It is clear that if $f\in S_{\alpha,\beta}$ we have that
    \begin{equation}\label{eqM3}
    c_k^{\alpha,\beta}({\mathcal L}_{\alpha,\beta}^{ -\gamma}f)=({\lambda}_k^{\alpha,\beta})^{-\gamma}c_k^{\alpha,\beta}(f), \quad k\in\Bbb N.
    \end{equation}
    Since ${\mathcal L}_{\alpha,\beta}^{ -\gamma}$ is bounded from $ L^{p(\cdot)}(0,\pi)$ into itself
    (see Proposition \ref{propM3}); for every $k\in\Bbb N$,
    ${\phi}_k^{\alpha,\beta}\in L^{p'(\cdot)}(0,\pi)=\left(L^{p(\cdot)}(0,\pi)\right)^*$ (\cite[Theorem 3.4.6]{DHHR})
    and $S_{\alpha,\beta}$ is dense in $L^{p(\cdot)}(0,\pi)$ (Proposition \ref{propS1}),
    we conclude that (\ref{eqM3}) holds for every $f\in L^{p(\cdot)}(0,\pi)$. Then, from Corollary \ref{corS2}
    we deduce that $f=0$ provided that ${\mathcal L}_{\alpha,\beta}^{ -\gamma}f=0$.
\end{proof}

By using Proposition~\ref{Prop:B1} we obtain the following characterization of the potential space $H_{\alpha,\beta}^{\gamma,p(\cdot)}(0,\pi)$.

\begin{Prop}\label{Prop:M5}
    Let $\gamma>0$ and $\alpha,\beta \geq -1/2$ such that $\alpha + \beta \neq -1$.
    Assume that $p \in{\mathcal B}(0,\pi)$. A function $f \in L^{p(\cdot)}(0,\pi)$ is in
    $H_{\alpha,\beta}^{\gamma,p(\cdot)}(0,\pi)$ if, and only if, the series
    $\sum_{n=0}^\infty (\lambda_n^{\alpha,\beta})^\gamma c_n^{\alpha,\beta}(f) \phi_n^{\alpha,\beta}$ converges in
    $L^{p(\cdot)}(0,\pi)$. Moreover, for every $f \in H_{\alpha,\beta}^{\gamma,p(\cdot)}(0,\pi)$,
    $$\|f\|_{H_{\alpha,\beta}^{\gamma,p(\cdot)}(0,\pi)}
        = \Big\| \sum_{n=0}^\infty (\lambda_n^{\alpha,\beta})^\gamma c_n^{\alpha,\beta}(f) \phi_n^{\alpha,\beta} \Big\|_{L^{p(\cdot)}(0,\pi)}.$$
\end{Prop}

\begin{proof}
    Let $f \in L^{p(\cdot)}(0,\pi)$. Suppose that $f \in H_{\alpha,\beta}^{\gamma,p(\cdot)}(0,\pi)$. Then, there exists
    $g \in L^{p(\cdot)}(0,\pi)$ such that $f=\mathcal{L}_{\alpha,\beta}^{-\gamma}g$. Thus, by \eqref{eqM3} we have that
    $c_n^{\alpha,\beta}(f)
        = (\lambda_n^{\alpha,\beta})^{-\gamma} c_n^{\alpha,\beta}(g)$, $n \in \N$. Hence,
    according to Proposition~\ref{Prop:B1}, the series
    $$\sum_{n=0}^\infty (\lambda_n^{\alpha,\beta})^\gamma c_n^{\alpha,\beta}(f) \phi_n^{\alpha,\beta}
        = \sum_{n=0}^\infty  c_n^{\alpha,\beta}(g) \phi_n^{\alpha,\beta}$$
    converges in $L^{p(\cdot)}(0,\pi)$.

    Assume now that the series $F=\sum_{n=0}^\infty (\lambda_n^{\alpha,\beta})^\gamma c_n^{\alpha,\beta}(f) \phi_n^{\alpha,\beta}$
    converges in $L^{p(\cdot)}(0,\pi)$. Then, by Proposition~\ref{propM4}, $\mathcal{L}_{\alpha,\beta}^{-\gamma}F=f$
    and $f \in H_{\alpha,\beta}^{\gamma,p(\cdot)}(0,\pi)$.
\end{proof}

As an immediate consequence of Proposition~\ref{Prop:M5} we establish the density of $S_{\alpha,\beta}$ in $H_{\alpha,\beta}^{\gamma,p(\cdot)}(0,\pi)$.

\begin{Cor}\label{Cor:M6}
    Let $\gamma>0$ and $\alpha,\beta \geq -1/2$ such that $\alpha + \beta \neq -1$.
    Assume that $p \in{\mathcal B}(0,\pi)$.
    Then, for every
    $f \in H_{\alpha,\beta}^{\gamma,p(\cdot)}(0,\pi)$,
    $$f = \lim_{n \to \infty} \sum_{k=0}^n  c_k^{\alpha,\beta}(f) \phi_k^{\alpha,\beta}, $$
    in the sense of convergence in $H_{\alpha,\beta}^{\gamma,p(\cdot)}(0,\pi)$.
\end{Cor}

\begin{proof}
    Let $f \in H_{\alpha,\beta}^{\gamma,p(\cdot)}(0,\pi)$. We have that $f=\mathcal{L}_{\alpha,\beta}^{-\gamma}g$, where
    $$g = \sum_{k=0}^\infty (\lambda_k^{\alpha,\beta})^\gamma  c_k^{\alpha,\beta}(f) \phi_k^{\alpha,\beta}, $$
    in the sense of convergence in $L^{p(\cdot)}(0,\pi)$. Then,
    $$\Big\| f - \sum_{k=0}^n c_k^{\alpha,\beta}(f) \phi_k^{\alpha,\beta} \Big\|_{H_{\alpha,\beta}^{\gamma,p(\cdot)}(0,\pi)}
        =  \Big\| g - \sum_{k=0}^n (\lambda_k^{\alpha,\beta})^\gamma c_k^{\alpha,\beta}(f) \phi_k^{\alpha,\beta} \Big\|_{L^{p(\cdot)}(0,\pi)}
        \longrightarrow 0, \quad \text{as } n \to \infty.$$
\end{proof}

\section{Boundedness of the higher order Riesz transforms} \label{sec:4}

This section has to do with the proof of assertions (b) and (c) of Section \ref{sec:intro}.

Firstly, we establish that $R_{\alpha,\beta}^k$ and $R_{\alpha,\beta}^{k,*}$ are composition of Jacobi Riesz transforms of order one.

\begin{Lem}\label{lemT1}
    Let $k\in\Bbb N$ and $\alpha,\beta\geq -1/2$ such that $\alpha+\beta \neq -1$. Then,
    \begin{equation}\label{eqT3}
        R_{\alpha,\beta}^kf
            =R_{\alpha+k-1,\beta+k-1}^1\circ R_{\alpha+k-2,\beta+k-2}^1\circ.....\circ R_{\alpha,\beta}^1f, \quad f\in S_{\alpha,\beta},
    \end{equation}
    and
    \begin{equation}\label{eqT4}
        R_{\alpha,\beta}^{k,*}f
            =R_{\alpha,\beta}^{1,*}\circ R_{\alpha+1,\beta+1}^{1,*}\circ.....\circ R_{\alpha+k-1,\beta+k-1}^{1,*}f, \quad f\in S_{\alpha+k,\beta+k}.
    \end{equation}
\end{Lem}

\begin{proof}
    We are going to prove (\ref{eqT3}), (\ref{eqT4}) can be shown in a similar way.
    It is sufficient to see that (\ref{eqT3}) is true when $f={\phi}_l^{\alpha,\beta}$, for every $l\in\Bbb N$.

    Let $l\in\Bbb N$. According to \cite[Lemma 3.1]{La} we have that
    \begin{equation}\label{eq:13.1}
        {\mathbb{D}}_{\alpha,\beta}^k{\phi}_l^{\alpha ,\beta}
            =(-1)^k\sqrt{(l-k+1)_k(l+\alpha+\beta+1)_k} \ {\phi}_{l-k}^{\alpha+k ,\beta+k}.
    \end{equation}
    Recall the definition of the Pochhammer symbol in \eqref{eq:9.1} and by convention ${\phi}_n^{\alpha,\beta}=0$, $n\in\Bbb Z$, $n<0$. Hence,
    \begin{equation*}\label{eqT5}
        R_{\alpha,\beta}^k{\phi}_l^{\alpha ,\beta}
            =(-1)^k \sqrt{\frac{(l-k+1)_k(l+\alpha+\beta+1)_k}{({\lambda}_l^{\alpha,\beta})^{k}}} \ {\phi}_{l-k}^{\alpha+k ,\beta+k}.
    \end{equation*}
    Since ${\lambda}_l^{\alpha,\beta}={\lambda}_{l-n}^{\alpha+n,\beta+n}$, $0\leq n\leq l$, we can write
    \begin{align*}
        R_{\alpha,\beta}^k{\phi}_l^{\alpha ,\beta}
            & = (-1)^k\prod_{n=0}^{k-1} \sqrt{ \frac{(l-n)(l+\alpha+\beta+1+n)}{{\lambda}_{l-n}^{\alpha+n,\beta+n}}} \ {\phi}_{l-k}^{\alpha+k ,\beta+k} \\
            & = (-1)^{k-1}\prod_{n=0}^{k-2} \sqrt{\frac{(l-n)(l+\alpha+\beta+1+n)}{\lambda_{l-n}^{\alpha+n,\beta+n}}} \ R_{\alpha+k-1,\beta+k-1}^1{\phi}_{l-k+1}^{\alpha+k-1 ,\beta+k-1} \\
            & =R_{\alpha+k-1,\beta+k-1}^1\circ R_{\alpha+k-2,\beta+k-2}^1\circ.....\circ R_{\alpha,\beta}^1{\phi}_l^{\alpha ,\beta},
    \end{align*}
    and (\ref{eqT3}) is established.
\end{proof}

We are going to prove that $R_{\alpha,\beta}^k$ and $R_{\alpha,\beta}^{k,*}$ define bounded operators
from $L^p_{w}(0,\pi)$ into itself for every $1<p<\infty$ and $w\in A_p(0,\pi)$.
As consecuence of the next lemma, we only need  to study the corresponding local operators (see \cite{BMTU} and \cite{CNS}).

We consider the domain $\mathcal{D}=\displaystyle\cup_{j=1}^4 \mathcal{D}_j$ represented in the figure bellow

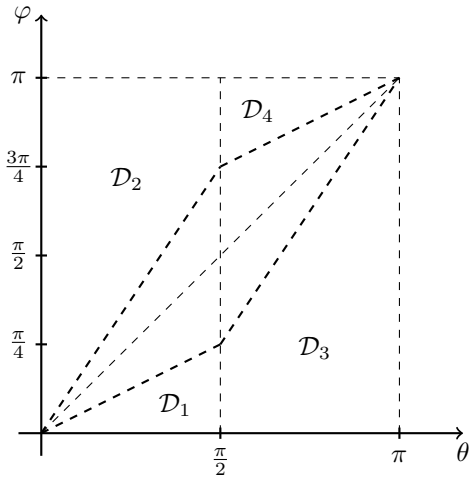
\begin{figure}[h!]
    \begin{minipage}{7.5cm}
        \begin{center}
            \begin{tikzpicture}[scale=1.5]
                \draw[->,thick] (-0.2,0) -- (3.7,0) node[below] {$\theta$};
                    \draw[->,thick] (0,-0.2) -- (0,3.7) node[left] {$\varphi$};

                    \draw[dashed] (0,pi) -- (pi,pi);
                    \draw[dashed] (pi,0) -- (pi,pi);
                    \draw[dashed] (0,0) -- (pi,pi);
                    \draw[dashed] (pi/2,0) -- (pi/2,pi);

                    \draw[dashed,thick] (0,0) -- (pi/2,pi/4);
                \draw[dashed,thick] (pi/2,pi/4) -- (pi,pi);
                \draw[dashed,thick] (0,0) -- (pi/2,3*pi/4);
                \draw[dashed,thick] (pi/2,3*pi/4) -- (pi,pi);

                \draw[-,thick] (pi, -0.05) -- (pi, 0.05);
                \node at (pi, -0.2) {$\pi$};
                \draw[-,thick] (pi/2, -0.05) -- (pi/2, 0.05);
                \node at (pi/2, -0.25) {$\frac{\pi}{2}$};

                \draw[-,thick] (-0.05,pi) -- (0.05,pi);
                \node at (-0.2,pi) {$\pi$};
                \draw[-,thick] (-0.05,pi/4) -- (0.05,pi/4);
                \node at (-0.2,pi/4) {$\frac{\pi}{4}$};
                \draw[-,thick] (-0.05,pi/2) -- (0.05,pi/2);
                \node at (-0.2,pi/2) {$\frac{\pi}{2}$};
                \draw[-,thick] (-0.05,3*pi/4) -- (0.05,3*pi/4);
                \node at (-0.2,3*pi/4) {$\frac{3\pi}{4}$};

                \node at (1.175,0.25) {$\mathcal{D}_1$};
                \node at (0.75,2.25) {$\mathcal{D}_2$};
                \node at (2.4,0.75) {$\mathcal{D}_3$};
                \node at (1.9,2.85) {$\mathcal{D}_4$};
            \end{tikzpicture}
        \end{center}
    \end{minipage}
    \begin{minipage}{7.5cm}
        \begin{align*}
            & \mathcal{D}_1=\Big\{(\theta,\varphi)\ : \ 0<\varphi <\frac{\theta}{2}, \ 0 < \theta < \frac{\pi}{2} \Big\}, \\
            & \mathcal{D}_2=\Big\{(\theta,\varphi)\ : \ 0< \frac{3\theta}{2}<\varphi<\pi, \ 0 < \theta < \frac{\pi}{2}\Big\}, \\
            & \mathcal{D}_3=\Big\{(\theta,\varphi)\ : \ 0<\varphi <\frac{3\theta-\pi}{2}, \ \frac{\pi}{2}<\theta<\pi\Big\}, \\
            & \mathcal{D}_4=\Big\{(\theta,\varphi)\ : \  \frac{\theta+\pi}{2}<\varphi<\pi, \ \frac{\pi}{2}<\theta<\pi\Big\}.
        \end{align*}
    \end{minipage}
\caption{Global regions}
\label{fig:regions}
\end{figure}

\begin{Lem}\label{lemT2}
    Suppose that $K:(0,\pi)\times(0,\pi)\backslash\{(\theta,\theta)\ : \ \theta\in(0,\pi)\}\  \longrightarrow \ \Bbb R$ is a measurable function such that
    $$|K(\theta,\varphi)|
        \leq{C\over{|\theta-\varphi|}}, \quad \theta,\varphi\in(0,\pi),\ \theta\neq\varphi.$$
    Then, for every $1<p<\infty$ and $w\in A_p(0,\pi)$ the operator $H$ defined by
    $$H f(\theta)
        =\int_0^\pi {K(\theta,\varphi) \chi_\mathcal{D}(\theta,\varphi) f(\varphi)d\varphi}, \quad \theta\in (0,\pi),$$
    is bounded from $L^p_{w}(0,\pi)$ into $L^p_{w}(0,\pi)$.
\end{Lem}

\begin{proof}
    We define
    $$H_j f(\theta)
        =\int_0^\pi K(\theta,\varphi) \chi_{\mathcal{D}_j}(\theta,\varphi) f(\varphi) d\varphi, \quad \theta\in (0,\pi),\ j=1,2,3,4.$$
    Thus, $\displaystyle H=\sum_{j=1}^4H_j$.

    By ${\mathcal M}$ we denote the Hardy-Littlewood maximal function on $(0,\pi)$. We have that
    \begin{align*}
        |H_1 f(\theta)|
            & \leq \int_0^{\theta/2}{{|f(\varphi)|}\over{|\theta-\varphi|}}d\varphi
              \leq \frac{C}{\theta}\int_0^{\theta/2}|f(\varphi)|d\varphi
                \leq C{\mathcal M}(f)(\theta), \quad \theta\in (0,\pi),
    \end{align*}
    and
    \begin{align*}
        |H_4 f(\theta)|
            & \leq \int_{(\theta+\pi)/2}^{\pi}{{|f(\varphi)|}\over{|\theta-\varphi|}}d\varphi
              \leq {C \over{\pi-\theta}}  \int_{\pi-3(\pi-\theta)/2}^{\pi}|f(\varphi)|d\varphi
              \leq C{\mathcal M}(f)(\theta), \quad \theta\in (0,\pi).
    \end{align*}
    By using the classical maximal theorem we deduce that $H_1$ and $H_4$ are bounded from
    $L^p_{w}(0,\pi)$ into itself, for every $1<p<\infty$ and $w\in A_p(0,\pi)$.

    The adjoint operator $H_2^*$ of $H_2$ is defined by
    $$H_2^* g(\varphi)
        = \chi_{(0,\frac{3\pi}{4})}(\varphi)\int_0^{2\varphi/3}{K(\theta,\varphi)g(\theta)d\theta}
        + \chi_{(\frac{3\pi}{4},\pi)}(\varphi)\int_0^{\pi/2}{K(\theta,\varphi)g(\theta)d\theta}, \quad \varphi \in (0,\pi).$$
    If $1<p<\infty$ and $w\in A_p(0,\pi)$, we deduce that
    \begin{align*}
        \| H_2^* g \|_{L^p_{w}(0,\pi)}
            & \leq C\left\{ \left(\int_0^{3\pi/4}w(\varphi)\left(\int_0^{2\varphi/3}{{|g(\theta)|}\over{|\theta-\varphi|}}d\theta\right)^p d\varphi\right)^{1/p}
                + \left(\int_{3\pi/4}^{\pi}w(\varphi)\left(\int_0^{\pi/2}{{|g(\theta)|}\over{|\theta-\varphi|}}d\theta\right)^pd\varphi\right)^{1/p}\right\} \\
            & \leq C \left\{ \left(\int_0^{\pi}w(\varphi)|{\mathcal M}(|g|)(\varphi)|^pd\varphi\right)^{1/p}
                          + \left(\int_{3\pi/4}^{\pi}w(\varphi) d\varphi \right)^{1/p}\int_0^{\pi}|g(\theta)|d\theta\right\} \\
            & \leq C\|g\|_{L^p_{w}(0,\pi)}, \quad g\in L^p_{w}(0,\pi).
    \end{align*}
    Hence, $H_2$ is bounded from $L^p_{w}(0,\pi)$ into itself for every $1<p<\infty$ and $w\in A_p(0,\pi)$.
    On the other hand, the adjoint operator $H_3^*$ of $H_3$ is given by
    $$H_3^* g(\varphi)
        =\chi_{(0,\frac{\pi}{4})}(\varphi)\int_{\pi/2}^{\pi}{K(\theta,\varphi)g(\theta)d\theta}
        +\chi_{(\frac{\pi}{4},\pi)}(\varphi)\int_{(2\varphi+\pi)/3}^{\pi}{K(\theta,\varphi)g(\theta)d\theta}.$$
    If $1<p<\infty$ and $w\in A_p(0,\pi)$, we get
    \begin{align*}
        \| H_3^* g \|_{L^p_{w}(0,\pi)}
            &\leq C \left\{ \left(\int_0^{\pi/4}w(\varphi)d\varphi\right)^{1/p}\int_0^{\pi}|g(\theta)|d\theta
                 + \left(\int_{\pi/4}^{\pi}w(\varphi)\left(\int_{(2\varphi+\pi)/3}^{\pi}{{|g(\theta)|}\over{|\theta-\varphi|}}d\theta\right)^pd\varphi\right)^{1/p}\right\} \\
            & \leq C\left\{\|g\|_{L^p_{w}(0,\pi)}
                    +\left(\int_{\pi/4}^{\pi}w(\varphi)\left({1\over{\pi - \varphi}} \int_{\pi-4(\pi-\varphi)/3}^{\pi}|g(\theta)|d\theta\right)^pd\varphi\right)^{1/p}\right\} \\
            & \leq C \left(\|g\|_{L^p_{w}(0,\pi)}+\|{\mathcal M}(|g|)\|_{L^p_{w}(0,\pi)}\right)
              \leq C\|g\|_{L^p_{w}(0,\pi)}, \quad g\in L^p_{w}(0,\pi).
    \end{align*}
    We conclude that $H_3$ is bounded from $L^p_{w}(0,\pi)$ into itself, for every $1<p<\infty$ and $w\in A_p(0,\pi)$.

    Thus, the proof of this lemma is finished.
\end{proof}

By using Lemmas \ref{lemT1} an \ref{lemT2} we will deduce the $L^p_{w}(0,\pi)$-boundedness of
$R_{\alpha,\beta}^m$ and $R_{\alpha,\beta}^{m,*}$ from the corresponding property of $R_{\alpha,\beta}^1$ and $R_{\alpha,\beta}^{1,*}$, respectively.

\begin{Prop}\label{propT2}
    Let $1<p<\infty$, $w\in A_p(0,\pi)$ and $\alpha,\beta\geq -1/2$ such that $\alpha+\beta\neq -1$.
    The Jacobi Riesz transforms $R_{\alpha,\beta}^1$
    and $R_{\alpha,\beta}^{1,*}$ define bounded operators  from $L^p_{w}(0,\pi)$ into itself.
\end{Prop}

We are going to use local Calder\'on-Zygmund theory for singular integrals (see \cite{CNS}).
We are inspired in the arguments developed by Nowak and Sj\"{o}gren in \cite{NoSj2}.

\begin{proof}[Proof of Proposition \ref{propT2}; the case of $R_{\alpha,\beta}^1$]
    By \eqref{eq:13.1} we have that
    $$R_{\alpha,\beta}^1 f
        =-\sum_{k=0}^{\infty} \sqrt{\frac{k(k+\alpha+\beta+1)}{\lambda_{k}^{\alpha,\beta}}}
                \ c_k^{\alpha,\beta}(f) \ {\phi}_{k-1}^{\alpha+1 ,\beta+1}, \quad f\in L^2(0,\pi).$$
    According to Plancherel's theorem, $R_{\alpha,\beta}^1$ is bounded from $L^2(0,\pi)$ into itself. By using \cite[Theorem 2.4]{ CNZ} we can write
    $$R_{\alpha,\beta}^1 f(\theta)
        =\lim_{\varepsilon\rightarrow 0^+}\int_{0, \ |\theta-\varphi|>\varepsilon}^{\pi}{R_{\alpha,\beta}^1(\theta,\varphi)f(\varphi)d\varphi},\ \ a.e.\ \theta\in (0,\pi),$$
    for every $f\in C_c^{\infty}(0,\pi)$. Here the kernel
    $R_{\alpha,\beta}^1(\theta,\varphi)$ is defined by
    $$R_{\alpha,\beta}^1(\theta,\varphi)
        = \int_0^{\infty} D_{\alpha,\beta} P_t^{\alpha,\beta}(\theta,\varphi)dt, \quad \theta,\varphi\in (0,\pi),\ \theta\neq\varphi.$$
    According to \cite[Theorem 2.4]{CNS} and Lemma \ref{lemT2}, to prove that $R_{\alpha,\beta}^1$ is bounded from
    $L^p_{w}(0,\pi)$ into itself, it is enough to show that
    \begin{equation}\label{eqT6}
        |R_{\alpha,\beta}^1(\theta,\varphi)|
            \leq{C\over{|\theta-\varphi|}}, \quad \theta,\varphi\in (0,\pi),\ \theta\neq\varphi,
    \end{equation}
    and
    \begin{equation}\label{eqT7}
        |\partial_{\theta}R_{\alpha,\beta}^1(\theta,\varphi)|+|\partial_{\varphi}R_{\alpha,\beta}^1(\theta,\varphi)|
            \leq{C\over{|\theta-\varphi|^2}}, \quad (\theta,\varphi)\in (0,\pi)^2\backslash \mathcal{D},\ \theta\neq\varphi,
    \end{equation}
    where $\mathcal{D}$ is the domain in Figure~\ref{fig:regions}.

    According to \cite[Proposition 4.1]{NoSj2} and \cite[(3)]{NoSj3} we have that for every $\theta,\varphi\in (0,\pi)$ and $t>0$,
    \begin{equation}\label{eq:15.1}
        P_t^{\alpha,\beta}(\theta,\varphi)
            =C_{\alpha,\beta}\left(\sin{{\theta}\over 2}\sin{{\varphi}\over 2}\right)^{\alpha+1/2}\left(\cos{{\theta}\over 2}\cos{{\varphi}\over 2}\right)^{\beta+1/2}\sinh{t\over 2}\int_{-1}^1\int_{-1}^1{{d\Pi_{\alpha}(u)d\Pi_{\beta}(v)}\over{(\cosh{t\over 2}-1+q(\theta,\varphi,u,v))^{\alpha+\beta+2}}},
    \end{equation}
    where
    $C_{\alpha,\beta}={{2^{-\alpha-\beta-1}}\over{\int_0^{\pi}(\sin{{\theta}\over 2})^{2\alpha+1}(\cos{{\theta}\over 2})^{2\beta+1}d\theta}}$,
    $d\Pi_{\alpha}(u)={{\Gamma(\alpha+1)}\over{\sqrt{\pi}\Gamma(\alpha+1/2)}}(1-u^2)^{\alpha-1/2}du$,
    and
    $$q(\theta,\varphi,u,v)=1-u\sin{{\theta}\over 2}\sin{{\varphi}\over 2}-v\cos{{\theta}\over 2}\cos{{\varphi}\over 2}.$$

    By proceeding as in \cite[Proof of Theorem 2.4; the case of $R_1^{\alpha,\beta}$]{NoSj2} and using
    \cite[Lemma 4.4 and trigonometric identities in p. 738]{NoSj2} we get that
    \begin{align}\label{eq:16.1}
        & |R_{\alpha,\beta}^1(\theta,\varphi)| \nonumber\\
        & \quad\leq C\int_0^{\infty}\sinh{t\over 2}\int_{-1}^1\int_{-1}^1{{\left(\sin{{\theta}\over 2}\sin{{\varphi}\over 2}\right)^{\alpha+1/2}\left(\cos{{\theta}\over 2}\cos{{\varphi}\over 2}\right)^{\beta+1/2}|\partial_{\theta}q(\theta,\varphi,u,v)|}\over{(\cosh{t\over 2}-1+q(\theta,\varphi,u,v))^{\alpha+\beta+3}}} d\Pi_{\alpha}(u)d\Pi_{\beta}(v)dt \nonumber \\
        & \quad \leq C \int_{-1}^1\int_{-1}^1{{\left(\sin{{\theta}\over 2}\sin{{\varphi}\over 2}\right)^{\alpha+1/2}\left(\cos{{\theta}\over 2}\cos{{\varphi}\over 2}\right)^{\beta+1/2}}\over{q^{\alpha+\beta+3/2}(\theta,\varphi,u,v)}}d\Pi_{\alpha}(u)d\Pi_{\beta}(v) \nonumber \\
        & \quad \leq C \int_{-1}^1{{\left(\sin{{\theta}\over 2}\sin{{\varphi}\over 2}\right)^{\alpha+1/2}\left(\cos{{\theta}\over 2}\cos{{\varphi}\over 2}\right)^{\beta+1/2} d\Pi_{\alpha}(u)}\over{(1-u\sin{{\theta}\over 2}\sin{{\varphi}\over 2})^{\beta+1/2}(1-u\sin{{\theta}\over 2}\sin{{\varphi}\over 2}-\cos{{\theta}\over 2}\cos{{\varphi}\over 2})^{\alpha+1}}} \nonumber \\
        & \quad \leq C  {{\left(\cos{{\theta}\over 2}\cos{{\varphi}\over 2}\right)^{\beta+1/2}}\over{(1-\sin{{\theta}\over 2}\sin{{\varphi}\over 2})^{\beta+1/2}}}{1\over{(1-\sin{{\theta}\over 2}\sin{{\varphi}\over 2}-\cos{{\theta}\over 2}\cos{{\varphi}\over 2})^{1/2}}}{{\left(\sin{{\theta}\over 2}\sin{{\varphi}\over 2}\right)^{\alpha+1/2}}\over{\left(1-\cos{{\theta}\over 2}\cos{{\varphi}\over 2}\right)^{\alpha+1/2}}} \nonumber \\
        & \quad \leq C \left({{\cos{{\theta}\over 2}\cos{{\varphi}\over 2}}\over{1-\sin{{\theta}\over 2}\sin{{\varphi}\over 2}-\cos{{\theta}\over 2}\cos{{\varphi}\over 2}+\cos{{\theta}\over 2}\cos{{\varphi}\over 2}}}\right)^{\beta+1/2}{1\over{\left(1-\cos{{\theta-\varphi}\over 2}\right)^{1/2}}} \nonumber \\
        & \quad \leq C \left({{\cos{{\theta}\over 2}\cos{{\varphi}\over 2}}\over{1-\cos{{\varphi-\theta}\over 2}+\cos{{\theta}\over 2}\cos{{\varphi}\over 2}}}\right)^{\beta+1/2}{1\over{|\theta-\varphi|}}\leq {C\over{|\theta-\varphi|}}, \quad \theta,\varphi\in (0,\pi).
    \end{align}
    Then (\ref{eqT6}) is proved.

    Also, we have that
    \begin{align}\label{eq:c1}
         \partial_{\theta}R_{\alpha,\beta}^1(\theta,\varphi)
            =&\left({{2\alpha+1}\over 4}{{\cos{{\theta}\over 2}}\over{\sin{{\theta}\over 2}}}-{{2\beta+1}\over 4}{{\sin{{\theta}\over 2}}\over{\cos{{\theta}\over 2}}}\right)R_{\alpha,\beta}^1(\theta,\varphi) \nonumber \\
        & +\left(\sin{{\theta}\over 2}\sin{{\varphi}\over 2}\right)^{\alpha+1/2}\left(\cos{{\theta}\over 2}\cos{{\varphi}\over 2}\right)^{\beta+1/2}T_{\alpha,\beta}(\theta,\varphi), \quad \theta,\varphi\in (0,\pi),
    \end{align}
    where
    $$T_{\alpha,\beta}(\theta,\varphi)
        =C_{\alpha,\beta}\partial_{\theta}^2\int_0^{\infty}\sinh{t\over 2}\int_{-1}^1\int_{-1}^1{{d\Pi_{\alpha}(u)d\Pi_{\beta}(v)}\over{(\cosh{t\over 2}-1+q(\theta,\varphi,u,v))^{\alpha+\beta+2}}}dt, \quad \theta,\varphi\in (0,\pi).$$
    We can write by \cite[Lemma 4.7]{NoSj2} and proceeding as in \cite[Proof of Theorem 2.4; the case of $R_N^{\alpha,\beta}$]{NoSj2},
    \begin{align}\label{eq:c2}
        & \left|\left(\sin{{\theta}\over 2}\sin{{\varphi}\over 2}\right)^{\alpha+1/2}\left(\cos{{\theta}\over 2}\cos{{\varphi}\over 2}\right)^{\beta+1/2}T_{\alpha,\beta}(\theta,\varphi)\right| \nonumber\\
        & \quad \leq C \int_{-1}^1\int_{-1}^1{{\left(\sin{{\theta}\over 2}\sin{{\varphi}\over 2}\right)^{\alpha+1/2}\left(\cos{{\theta}\over 2}\cos{{\varphi}\over 2}\right)^{\beta+1/2}}\over{q(\theta,\varphi,u,v)^{\alpha+\beta+2}}}d\Pi_{\alpha}(u)d\Pi_{\beta}(v)
            \leq {C\over{|\theta-\varphi|^2}},  \quad \theta,\varphi\in (0,\pi).
    \end{align}

    On the other hand
    \begin{align}\label{eq:17.1}
        {{\cos{{\theta}\over 2}}\over{\sin{{\theta}\over 2}}}
            & = {{\cos{{\theta}\over 2}\cos{{\varphi}\over 2}}\over{\sin{{\theta}\over 2}\cos{{\varphi}\over 2}}}
              = {{\cos{{\theta}\over 2}\cos{{\varphi}\over 2}}\over{\sin{{\theta}\over 2}\cos{{\varphi}\over 2}-\sin{{\varphi}\over 2}\cos{{\theta}\over 2}+\sin{{\varphi}\over 2}\cos{{\theta}\over 2}}}
              = {{\cos{{\theta}\over 2}\cos{{\varphi}\over 2}}\over{\sin{{\theta-\varphi}\over 2}+\sin{{\varphi}\over 2}\cos{{\theta}\over 2}}}\nonumber \\
            & \leq {1\over{\sin{{\theta-\varphi}\over 2}}},  \quad 0<\varphi <\theta<\pi.
    \end{align}
    If $\varphi\in (0,\pi)$, $\theta\in (0,\pi/2)$ and $\theta<\varphi<3\theta/2$, then $\sin\varphi/3<\sin\theta/2$ and
    \begin{align}\label{eq:17.2}
        {{\cos{{\theta}\over 2}}\over{\sin{{\theta}\over 2}}}
            & \leq{{\cos{{\theta}\over 2}}\over{\sin{{\varphi}\over 3}}}
             = {{\cos{{\theta}\over 2}\cos{{\theta}\over 3}}\over{\sin{{\varphi}\over 3}\cos{{\theta}\over 3}-\sin{{\theta}\over 3}\cos{{\varphi}\over 3}+\sin{{\theta}\over 3}\cos{{\varphi}\over 3}}}
             \leq {{1}\over{\sin{{\theta-\varphi}\over 3}}}.
    \end{align}
    Also, we get
    \begin{align}\label{eq:17.3}
        & {{\cos{{\theta}\over 2}}\over{\sin{{\theta}\over 2}}}
        \leq \frac{1}{\sin \frac{\pi}{4}}
        \leq {C\over{\sin|{{\theta-\varphi}\over 2}|}}, \quad 0<\varphi<\pi,\ \ \pi/2<\theta<\pi.
    \end{align}
    By combining \eqref{eq:16.1}, \eqref{eq:17.1}, \eqref{eq:17.2} and \eqref{eq:17.3} we obtain
    \begin{equation}\label{eq:c3}
        \left|{{\cos{{\theta}\over 2}}\over{\sin{{\theta}\over 2}}}R_{\alpha,\beta}^1(\theta,\varphi)\right|
            \leq{C\over{|\theta-\varphi|^2}}, \quad (\theta,\varphi)\in (0,\pi)^2\backslash \mathcal{D}.
    \end{equation}

    We can write
    \begin{equation}\label{eq:17.4}
        {{\sin{{\theta}\over 2}}\over{\cos{{\theta}\over 2}}}
            =-{{\cos{{\pi - \theta}\over 2}}\over{\sin{{\pi - \theta}\over 2}}}, \quad \theta \in (0,\pi),
    \end{equation}
    and by symmetries reasons and proceeding as above we get
    \begin{equation}\label{eq:c4}
        \left|{{\sin{{\theta}\over 2}}\over{\cos{{\theta}\over 2}}}R_{\alpha,\beta}^1(\theta,\varphi)\right|
            \leq{C\over{|\theta-\varphi|^2}}, \quad (\theta,\varphi)\in (0,\pi)^2\backslash \mathcal{D}.
    \end{equation}
    From \eqref{eq:c1}, \eqref{eq:c2}, \eqref{eq:c3} and \eqref{eq:c4}  we conclude that
    $$|\partial_{\theta}R_{\alpha,\beta}^1(\theta,\varphi)|
        \leq{C\over{|\theta-\varphi|^2}}, \quad (\theta,\varphi)\in (0,\pi)^2\backslash \mathcal{D}.$$
    In a similar way, we can see that
    $$|\partial_{\varphi}R_{\alpha,\beta}^1(\theta,\varphi)|
        \leq{C\over{|\theta-\varphi|^2}}, \quad (\theta,\varphi)\in (0,\pi)^2\backslash \mathcal{D}.$$
    Thus, (\ref{eqT7}) is established.
\end{proof}

\begin{proof}[Proof of Proposition \ref{propT2}; the case of $R_{\alpha,\beta}^{1,*}$]
    We have that
    $$R_{\alpha,\beta}^{1,*}f
        =-\sum_{k=0}^{\infty}
            \sqrt{\frac{(k+1)(k+\alpha+\beta+2)}{\lambda_k^{\alpha+1,\beta+1}}}
            \ c_k^{\alpha+1,\beta+1}(f) \ {\phi}_{k+1}^{\alpha ,\beta}, \quad f\in L^2(0,\pi).$$
    From Plancherel's theorem we deduce that $R_{\alpha,\beta}^{1,*}$ is a bounded operator from $L^2(0,\pi)$ into itself.

    If $f\in C_c^{\infty}(0,\pi)$, then for every $m\in\Bbb N$ there exists $C_m$ such that
    \begin{equation*}\label{eqT8}
        |c_k^{\alpha+1,\beta+1}(f)|
            \leq C_m(k+1)^{-m}, \quad k\in\Bbb N.
    \end{equation*}
    Suppose that  $f,g\in C_c^{\infty}(0,\pi)$. Partial integration leads to
    $$\int_0^{\pi}R_{\alpha,\beta}^{1,*}f(\theta)g(\theta)d\theta
        =\int_0^{\pi}f(\varphi){\mathcal L}_{\alpha+1,\beta+1}^{-1/2}\left(D_{\alpha,\beta}g\right)(\varphi)d\varphi.$$
    By taking into account the rapid decay of the sequence $\left(c_k^{\alpha,\beta}(g)\right)_{k\in\Bbb N}$ and \cite[Lemma 3.1]{La} we write
    $$D_{\alpha,\beta}g(\theta)
        =-\sum_{k=0}^{\infty}\sqrt{k(k+\alpha+\beta+1)}c_k^{\alpha,\beta}(g){\phi}_{k-1}^{\alpha+1 ,\beta+1}(\theta), \quad \theta\in (0,\pi),$$
    and
    $${\mathcal L}_{\alpha+1,\beta+1}^{-1/2}\left(D_{\alpha,\beta}g\right)(\theta)
        =-\sum_{k=0}^{\infty} \sqrt{\frac{k(k+\alpha+\beta+1)}{\lambda_k^{\alpha,\beta}}}
                \ c_k^{\alpha,\beta}(g) \ {\phi}_{k-1}^{\alpha+1 ,\beta+1}(\theta)
        =R_{\alpha,\beta}^1 g(\theta),\ \ \theta\in (0,\pi).$$
    Hence, $R_{\alpha,\beta}^{1,*}$ is the adjoint of $R_{\alpha,\beta}^1$ (fact justifying the notation).
    Thus, $R_{\alpha,\beta}^{1,*}$ defines a bounded operator from $L^p_{w}(0,\pi)$ into itself, for every $1<p<\infty$ and $w\in A_p(0,\pi)$.
\end{proof}

Combining \cite[Theorem 1.3]{CFMP} with Lemma~\ref{lemT1} and Proposition~\ref{propT2} we obtain the following.

\begin{Prop}\label{prop:RieszLp(.)}
    Let $k\in\Bbb N$ and $\alpha,\beta\geq -1/2$ such that $\alpha+\beta \neq -1$.
    Suppose that $p \in{\mathcal B}(0,\pi)$.
    Then, $R_{\alpha,\beta}^k$ and $R_{\alpha,\beta}^{k,*}$ define bounded operators from $L^{p(\cdot)}(0,\pi)$ into itself.
\end{Prop}

According to \cite[Lemma 3.1]{La} we get, for every $f\in S_{\alpha,\beta}$,
\begin{equation*}\label{eqT9}
    R_{\alpha,\beta}^{k,*}R_{\alpha,\beta}^kf
        =\sum_{n=k}^{\infty}{{(n-k+1)_k(n+\alpha+\beta+1)_k}\over{(\lambda_n^{\alpha,\beta})^k}}
        \ c_k^{\alpha,\beta}(f) \ {\phi}_n^{\alpha,\beta}.
\end{equation*}
Notice that, for every $n\in\Bbb N$, $n\geq k$,
$$(n-k+1)_k
    =\Big(\sqrt{{\lambda}_n^{\alpha,\beta}}-\sqrt{{\lambda}_{k-1}^{\alpha,\beta}}\Big)
     \Big(\sqrt{{\lambda}_n^{\alpha,\beta}}-\sqrt{{\lambda}_{k-2}^{\alpha,\beta}}\Big)
    \cdot ...\cdot
    \Big(\sqrt{{\lambda}_n^{\alpha,\beta}}-\sqrt{{\lambda}_0^{\alpha,\beta}}\Big),$$
and
$$(n+\alpha+\beta+1)_k
    =\Big(\sqrt{{\lambda}_n^{\alpha,\beta}}+\sqrt{{\lambda}_0^{\alpha,\beta}}\Big)
     \Big(\sqrt{{\lambda}_n^{\alpha,\beta}}+\sqrt{{\lambda}_1^{\alpha,\beta}}\Big)
     \cdot...\cdot
     \Big(\sqrt{{\lambda}_n^{\alpha,\beta}}+\sqrt{{\lambda}_{k-1}^{\alpha,\beta}}\Big).$$
We consider the function $M$ given by
$$M(x)
    ={{x^k}\over{\displaystyle \prod_{j=0}^{k-1}\Big(x-{\lambda}_j^{\alpha,\beta}}\Big)}, \qquad x\neq \lambda_j^{\alpha,\beta},\ j=0,...,k-1,$$
and we choose a smooth function $\phi$ on $(0,\infty)$ such that
$$\phi(x)
    =\left\{\begin{array}{l}
                0, \quad 0<x< {\lambda}_{k-1}^{\alpha,\beta}+{{\alpha+\beta+1}\over8}, \\
                \\
                1, \quad  x\geq {\lambda}_{k}^{\alpha,\beta}-{{\alpha+\beta+1}\over 8}.
            \end{array}\right.$$
Take $m=\phi M$. Then,
\begin{equation}\label{eqT10}
    m({\mathcal L}_{\alpha,\beta})R_{\alpha,\beta}^{k,*}R_{\alpha,\beta}^kf
        =f, \quad f\in S_{\alpha,\beta}.
\end{equation}
It is not hard to see that $m$ satisfies condition \eqref{eq:Mihlin} of proposition \ref{propM1}.
Hence, by Proposition~\ref{propM2} (with $a=0$) we infer the following.

\begin{Prop}\label{propT3}
    Let $\alpha,\beta\geq -1/2$ such that $\alpha+\beta \neq -1$.
    Suppose that $p \in{\mathcal B}(0,\pi)$.
    Then, the Jacobi spectral multiplier $m({\mathcal L}_{\alpha,\beta})$, where $m=\phi M$ is as above,
    defines a bounded operator from $L^{p(\cdot)}(0,\pi)$ into itself.
\end{Prop}

\section{Proof of Theorem~\ref{ThZ}} \label{sec:7}

First of all we establish the following lemma where we define some Jacobi spectral multipliers that will be useful in the sequel.
\begin{Lem}\label{lemZ1}
    Let $\varepsilon, \gamma >0$, $r\in\Bbb N$ with $r>\gamma$ and $\alpha,\beta\geq -1/2$ such that $\alpha + \beta \neq -1$.
    Assume that $p \in{\mathcal B}(0,\pi)$. We define, for each $t>0$, the functions
    \begin{equation}\label{eq:multipliers}
        Y_{\varepsilon}(t)=(1-e^{-\varepsilon t})^r, \quad
       M_{\varepsilon}(t)={{(1-e^{-\varepsilon t})^r}\over{(\varepsilon t)^{\gamma/2}}} \quad \text{and} \quad
        H_\varepsilon(t)=\int_{\varepsilon t}^\infty{{(1-e^{-u})^r}\over{u^{1+\gamma}}}du.
    \end{equation}
    By $m_\varepsilon$ we represent $Y_\varepsilon$, $M_\varepsilon$ or $H_\varepsilon$.
    Then, $m_\varepsilon$ defines a Jacobi spectral multiplier on $L^{p(\cdot)}(0,\pi)$.
    Moreover,
    $$\sup_{\varepsilon >0}\|m_\varepsilon(\mathcal{L}_{\alpha,\beta})\|_{L^{p(\cdot)}(0,\pi)}
        < \infty.$$
\end{Lem}

\begin{proof}
    Straightforward manipulations allow us to show that, for every $\ell \in\Bbb N$, there exists $C>0$ such that
    $$\sup_{\varepsilon >0}\big |t^\ell{{d^\ell}\over{dt^\ell}}m_\varepsilon(t)\big |
        \leq C,$$
    where $C$ does not depend on $\varepsilon$. Then, by Proposition~\ref{propM2} (taken with $a=0$) we concluded the desired results.
\end{proof}

\begin{Prop}\label{prop:Iepsgamr}
    Let $\varepsilon, \gamma >0$, $r\in\Bbb N$ with $r>\gamma$ and $\alpha,\beta\geq -1/2$ such that $\alpha + \beta \neq -1$.
    Assume that $p \in{\mathcal B}(0,\pi)$.
    Then, the operator $I_{\varepsilon}^{\gamma,r}$ defined in \eqref{eqH1} is bounded from $L^{p(\cdot)}(0,\pi)$ into itself.
\end{Prop}

\begin{proof}
    Let $f\in S_{\alpha,\beta}$. We can write
    $$(I-W_u^{\alpha,\beta})^rf
        =\sum_{n=0}^{\infty} Y_u\left({\lambda}_n^{\alpha,\beta}\right) c_n^{\alpha,\beta}(f){\phi}_n^{\alpha,\beta}
        = Y_u(\mathcal{L}_{\alpha,\beta}) f, \quad u>0,$$
    where the series is actually a finite sum. According to Lemma \ref{lemZ1}, we deduce that,
    \begin{align*}
        \|I_{\varepsilon}^{\gamma,r} f\|_{L^{p(\cdot)}(0,\pi)}
            \leq C \sup_{u>0}\| Y_u(\mathcal{L}_{\alpha,\beta}) f\|_{L^{p(\cdot)}(0,\pi)} \int_{\varepsilon}^{\infty} \frac{du}{u^{\gamma+1}}
            \leq C \|f\|_{L^{p(\cdot)}(0,\pi)}.
    \end{align*}
  Taking into account that $S_{\alpha,\beta}$ is a dense subspace of $L^{p(\cdot)}(0,\pi)$ (Proposition~\ref{propS1}) the conclusion follows.
\end{proof}

\begin{proof}[Proof of Theorem~\ref{ThZ}]
    Suppose that $f\in D_{p(\cdot)}({\mathcal L}_{\alpha,\beta}^\gamma)$ and call $g=\displaystyle\lim_{\varepsilon\rightarrow 0^+} I_\varepsilon^{\gamma,r} f$.
    Since ${\mathcal L}_{\alpha,\beta}^{-\gamma}$ is a bounded operator from $L^{p(\cdot)}(0,\pi)$ into itself (Proposition \ref{propM3}), we have that
    \begin{align*}
    {\mathcal L}_{\alpha,\beta}^{-\gamma}g
        & =  C_{\gamma,r}\lim_{\varepsilon\rightarrow 0^+} {\mathcal L}_{\alpha,\beta}^{-\gamma}\int_{\varepsilon}^{\infty}{{(I-W_u^{\alpha,\beta})^rf}\over{u^{1+\gamma}}}du \\
        & = C_{\gamma,r}\lim_{\varepsilon\rightarrow 0^+}\int_{\varepsilon}^{\infty}{{(I-W_u^{\alpha,\beta})^r}\over{u^{1+\gamma}}}{\mathcal L}_{\alpha,\beta}^{-\gamma}f\ du,
        \quad \text{in } L^{p(\cdot)}(0,\pi).
    \end{align*}
    We can write
    \begin{align*}
        {{(I-W_u^{\alpha,\beta})^rf}\over{u^{\gamma/2}}}{\mathcal L}_{\alpha,\beta}^{-\gamma}f
            = & \sum_{n=0}^{\infty} {{\Big(1-e^{-u{\lambda}_n^{\alpha,\beta}}\Big)^r}\over{\Big(u{\lambda}_n^{\alpha,\beta}\Big)^{\gamma/2}\left({\lambda}_n^{\alpha,\beta}\right)^{\gamma/2}}}
                \ c_n^{\alpha,\beta}(f)  \ {\phi}_n^{\alpha,\beta}
            = M_u(\mathcal{L}_{\alpha,\beta}) {\mathcal L}_{\alpha,\beta}^{-\gamma/2}f, \quad u>0,
    \end{align*}
    where $M_u$ was defined in \eqref{eq:multipliers}.
    According to Lemma \ref{lemZ1} and Propositions \ref{Prop:B1} and \ref{propM3}, there exists $C>0$ such that
    \begin{align*}
        \Big\| M_u(\mathcal{L}_{\alpha,\beta}) {\mathcal L}_{\alpha,\beta}^{-\gamma/2} \Big( \sum_{n=0}^\ell c_n^{\alpha,\beta}(f){\phi}_n^{\alpha,\beta} \Big) \Big\|_{L^{p(\cdot)}(0,\pi)}
            & \leq C \Big\|\sum_{n=0}^\ell c_n^{\alpha,\beta}(f){\phi}_n^{\alpha,\beta} \Big\|_{L^{p(\cdot)}(0,\pi)} \\
            &  \leq C\|f\|_{L^{p(\cdot)}(0,\pi)}, \quad \ell \in\Bbb N\ \mbox{and}\ u>0.
    \end{align*}
    Also, since $u^{-1-\gamma}\in L^1(\varepsilon,\infty)$, $\varepsilon>0$, we obtain
    \begin{align*}
        {\mathcal L}_{\alpha,\beta}^{-\gamma}\int_{\varepsilon}^{\infty}{{(I-W_u^{\alpha,\beta})^rf}\over{u^{1+\gamma}}}du
            & = \sum_{n=0}^\infty \int_{\varepsilon}^{\infty}{{\left(1-e^{-u{\lambda}_n^{\alpha,\beta}}\right)^r}\over{\left(u{\lambda}_n^{\alpha,\beta}\right)^{\gamma}}}{{du}\over u}
                    \ c_n^{\alpha,\beta}(f) \ {\phi}_n^{\alpha,\beta} \\
            & = \sum_{n=0}^\infty \int_{\varepsilon{\lambda}_n^{\alpha,\beta}}^{\infty}{{\left(1-e^{-u}\right)^r}\over{u^{1+\gamma}}}du
                    \ c_n^{\alpha,\beta}(f) \ {\phi}_n^{\alpha,\beta}
              = H_\varepsilon(\mathcal{L}_{\alpha,\beta}) f, \quad \varepsilon>0,
    \end{align*}
    where $H_\varepsilon$ was defined in \eqref{eq:multipliers}.

    Suppose that $F\in S_{\alpha,\beta}$. We can write, for every $l\in\Bbb N$,
    \begin{align*}
       \lim_{\varepsilon\rightarrow 0^+} H_\varepsilon(\mathcal{L}_{\alpha,\beta}) F
            & =  \lim_{\varepsilon\rightarrow 0^+} \sum_{n=0}^\ell \int_{\varepsilon{\lambda}_n^{\alpha,\beta}}^{\infty}{{\left(1-e^{-u}\right)^r}\over{u^{1+\gamma}}}du\ c_n^{\alpha,\beta}(F){\phi}_n^{\alpha,\beta}
             = \frac{1}{C_{\gamma,r}} \sum_{n=0}^\ell c_n^{\alpha,\beta}(F){\phi}_n^{\alpha,\beta}
            ={F\over{C_{\gamma,r}}},
    \end{align*}
    in the sense of convergence in $L^{p(\cdot)}(0,\pi)$.
    Since $S_{\alpha,\beta}$ is dense in $L^{p(\cdot)}(0,\pi)$ (Proposition \ref{propS1}), Lemma \ref{lemZ1} leads to

    $$\lim_{\varepsilon\rightarrow 0^+}H_\varepsilon(\mathcal{L}_{\alpha,\beta})(f)={f\over{C_{\gamma,r}}}.$$
    Thus, we conclude that ${\mathcal L}_{\alpha,\beta}^{-\gamma}g=f$.

    On the other hand, take $f\in H_{\alpha,\beta}^{\gamma,p(\cdot)}(0,\pi)$ such that $f={\mathcal L}_{\alpha,\beta}^{-\gamma}g$,
    with $g\in L^{p(\cdot)}(0,\pi)$. Then, as it has just been proved,
    $$\lim_{\varepsilon\rightarrow 0^+} I_\varepsilon^{\gamma,r}f
        = C_{\gamma,r}\lim_{\varepsilon\rightarrow 0^+}\int_{\varepsilon}^{\infty}{{(I-W_u^{\alpha,\beta})^r}\over{u^{1+\gamma}}}{\mathcal L}_{\alpha,\beta}^{-\gamma} g \ du
        = g,$$
    in the sense of convergence in $L^{p(\cdot)}(0,\pi)$.
\end{proof}

\begin{Rem}\label{remZ1}
    A careful reading of the above proof reveals that we can consider any
    $r\in\Bbb N$, $r>\gamma$ (not necessarily $r<\gamma\leq r+1$). This fact implies that the operator
    ${\mathcal L}_{\alpha,\beta}^{\gamma}$ can be defined by \eqref{eqH2}, for any $r\in\Bbb N$, $r>\gamma$.
\end{Rem}

\section{Proof of Theorem~\ref{Th2}} \label{sec:5}

Assume that $\gamma>0$. It is not hard to see that $\partial_t^\gamma e^{-at} =e^{i\pi\gamma} a^\gamma e^{-at}$, $t,a >0$.
Thus, we have that, for every $f \in S_{\alpha,\beta} \cup C_c^\infty(0,\pi)$,
$$\partial_t^\gamma P_t^{\alpha,\beta}f(\theta)
    = \sum_{n=0}^\infty e^{i\pi\gamma} (\lambda_n^{\alpha,\beta})^{\gamma/2} e^{-t \sqrt{\lambda_n^{\alpha,\beta}}}
        c_n^{\alpha,\beta}(f) \phi_n^{\alpha,\beta}(\theta), \quad \theta \in (0,\pi).$$
Hence, for every $f \in S_{\alpha,\beta} \cup C_c^\infty(0,\pi)$,
$$g_{\alpha,\beta}^\gamma (f)(\theta)
    < \infty, \quad \theta \in (0,\pi).$$
Our first objective is to establish $L_{\omega}^p$-boundedness properties of $g_{\alpha,\beta}^\gamma$-functions.

\begin{Prop}\label{PropG1}
    Let $\gamma>0$ and $\alpha, \beta \geq -1/2$. Then, $g_{\alpha,\beta}^\gamma$ defines a bounded
    (quasi-linear) operator from $L^p_w(0,\pi)$ into itself, for every $1<p<\infty$ and $w \in A_p(0,\pi)$.
\end{Prop}

\begin{proof}
    For every $N \in \N$, we define
    $$\mathcal{G}_{\alpha,\beta}^{\gamma,N}(f)(\theta)
        = \left( \int_{1/N}^N \Big| t^{\gamma} \partial_t^\gamma P_t^{\alpha, \beta} (f) (\theta) \Big|^2 \frac{dt}{t} \right)^{1/2}, \quad \theta \in (0,\pi).$$
    We will show that, for every $1<p<\infty$ and $w \in A_p(0,\pi)$, there exists $C>0$ independent of $N \in \N$, such that
    \begin{equation}\label{eq:G1}
        \| \mathcal{G}_{\alpha,\beta}^{\gamma,N}(f)\|_{L^p_w(0,\pi)}
            \leq C \| f \|_{L^p_w(0,\pi)}, \quad f \in  L^p_w(0,\pi).
    \end{equation}
    From \eqref{eq:G1}, by using monotone convergence theorem, we deduce that for every $1<p<\infty$ and $w \in A_p(0,\pi)$,
    there exists $C>0$ satisfying that
    $$ \|g_{\alpha,\beta}^{\gamma}(f)\|_{L^p_w(0,\pi)}
            \leq C \| f \|_{L^p_w(0,\pi)}, \quad f \in  L^p_w(0,\pi).$$
    In order to show \eqref{eq:G1} we apply the local Calder\'on-Zygmund theory \cite{CNS} in a Banach valued setting
    \cite{RRT}.

    By proceeding as in \cite[Proposition 2.1]{BFRTT2} we obtain
    \begin{equation}\label{eq:G2}
        \frac{2^{2\gamma}}{\Gamma(2\gamma)} \int_0^\pi \int_0^\infty
            t^\gamma \partial_t^\gamma P_t^{\alpha,\beta}f(\theta) t^\gamma \partial_t^\gamma P_t^{\alpha,\beta}(\bar{g})(\theta) d\theta \frac{dt}{t}
        = \int_0^\pi f(\theta) \bar{g}(\theta) d\theta, \quad f,g \in S_{\alpha,\beta}.
    \end{equation}
    Thus, for every $N \in \N$, we get
    \begin{equation}\label{eq:G3}
        \| \mathcal{G}_{\alpha,\beta}^{\gamma,N}(f)\|_{L^2(0,\pi)}^2
            = \frac{\Gamma(2\gamma)}{2^{2\gamma}} \| f \|_{L^2(0,\pi)}^2, \quad f \in S_{\alpha,\beta}.
    \end{equation}
    Hence, $g_{\alpha,\beta}^{\gamma}$ and $\mathcal{G}_{\alpha,\beta}^{\gamma,N}$, $N \in \N$, can be extended from $S_{\alpha,\beta}$ to
    $L^2(0,\pi)$ as a bounded operators from  $L^2(0,\pi)$ into itself.

    Let $m \in \N$. According to \cite[Lemma 4]{BCCFR1} we have that
    \begin{equation}\label{eq:20.1}
        \Big| \partial_t^m[te^{-t^2/4u}] \Big|
            \leq C e^{-t^2/4u} u^{(1-m)/2}, \quad t,u \in (0,\infty).
    \end{equation}
    By \eqref{eq:3.1.1} and by taking into account that \eqref{eqS2} and \eqref{eq:20.1} the differentiation under the integral sign is justified, so
    we can write
    $$\partial_t^m P_t^{\alpha,\beta}(\theta,\varphi)
        = \frac{1}{\sqrt{4\pi}} \int_0^\infty \partial_t^m \Big[ \frac{t e^{-t^2/4u}}{u^{3/2}} \Big]
            W_u^{\alpha,\beta}(\theta,\varphi) du, \quad t>0 \text{ and } \theta, \varphi \in (0,\pi).$$

    From \eqref{eqS2} and \eqref{eq:20.1} it follows that
    \begin{align}\label{eq:G4}
        \Big| \partial_t^m P_t^{\alpha,\beta}(\theta,\varphi) \Big|
            \leq & C \int_0^\infty \frac{e^{-c(t^2+(\theta-\varphi)^2)/u}}{u^{(m+3)/2}} du \nonumber \\
            \leq & \frac{C}{(t^2+(\theta-\varphi)^2)^{(m+1)/2}}, \quad t>0 \text{ and } \theta, \varphi \in (0,\pi).
    \end{align}
    Let $f \in L^2(0,\pi)$. By \eqref{eq:G4} we obtain
    $$ \partial_t^m P_t^{\alpha,\beta}f(\theta)
        = \int_0^\pi \partial_t^m P_t^{\alpha, \beta}(\theta, \varphi)f(\varphi) d\varphi, \quad t>0 \text{ and } \theta \in (0,\pi). $$
    Thus, if $m-1 \leq \gamma < m$, \eqref{eq:G4} leads to
    \begin{align*}
        \Big| \partial_t^\gamma P_t^{\alpha,\beta}f(\theta) \Big|
            & \leq C \int_0^\infty \int_0^\pi \Big| \partial_t^m P_{t+s}^{\alpha,\beta}(\theta,\varphi) \Big| \ |f(\varphi)| d\varphi s^{m-\gamma-1} ds \\
            & \leq C \int_0^\infty \int_0^\pi \frac{|f(\varphi)|}{[(t+s)^2+(\theta-\varphi)^2]^{(m+1)/2} }  d\varphi s^{m-\gamma-1} ds \\
            & \leq C \int_0^\infty  \frac{s^{m-\gamma-1}}{(t+s)^{m+1}} ds \ \|f\|_{L^2(0,\pi)}
              \leq \frac{C}{t^{\gamma+1}} \|f\|_{L^2(0,\pi)}, \quad t>0 \text{ and } \theta \in (0,\pi).
    \end{align*}
    Hence, we obtain, for every $N \in \N$,
    \begin{equation}\label{eq:G5}
        \mathcal{G}_{\alpha,\beta}^{\gamma,N}(f)(\theta)
            \leq C  \Big( \int_{1/N}^N \frac{dt}{t^3} \Big)^{1/2} \|f\|_{L^2(0,\pi)}, \quad \theta \in (0,\pi).
    \end{equation}
    This estimate shows that, for every $N \in \N$, $\mathcal{G}_{\alpha,\beta}^{\gamma,N}$ is a bounded operator from $L^2(0,\pi)$ into itself.
    By \eqref{eq:G3} we conclude that, for every $N \in \N$,
    \begin{equation}\label{eq:G6}
        \|\mathcal{G}_{\alpha,\beta}^{\gamma,N}(f)\|_{L^2(0,\pi)}^2
            = \frac{\Gamma(2\gamma)}{2^{2\gamma}} \|f\|_{L^2(0,\pi)}^2, \quad f \in L^2(0,\pi).
    \end{equation}
    Note that \eqref{eq:G6}, in contrast with \eqref{eq:G5}, shows that the family $\{\mathcal{G}_{\alpha,\beta}^{\gamma,N}\}_{N \in \N}$ is bounded
    in $\mathcal{L}(L^2(0,\pi))$, the space of bounded operators from $L^2(0,\pi)$ into itself.

    Let $N \in \N$. We consider the operator
    $$T_{\alpha, \beta}^{\gamma, N}(f)(\theta)
        = \int_0^\pi K_{\alpha,\beta}^{\gamma, N}(\theta, \varphi) f(\varphi) d\varphi,$$
    where, for every $\theta, \varphi \in (0,\pi)$, $\theta \neq \varphi$,
    $$[K_{\alpha,\beta}^{\gamma, N}(\theta,\varphi)](t)
        = t^\gamma \partial_t^\gamma P_t^{\alpha,\beta}(\theta,\varphi), \quad t \in (1/N,N),$$
    and the integral is understood in the $L^2((1/N,N), dt/t)$-B\"ochner sense.

    From \eqref{eq:G4} we deduce that
    \begin{align}\label{eq:G7}
        \Big\| K_{\alpha,\beta}^{\gamma, N}(\theta,\varphi) \Big\|_{L^2((1/N,N), dt/t)}
            & \leq C \Big( \int_{1/N}^N \Big| t^\gamma \int_0^\infty \frac{s^{m-\gamma-1}}{((t+s)^2+(\theta-\varphi)^2)^{(m+1)/2}} ds\Big|^2 \frac{dt}{t} \Big)^{1/2} \nonumber\\
            & \leq C  \Big( \int_{1/N}^N \frac{t^{2\gamma-1}}{(t + |\theta-\varphi|)^{2\gamma+2}} dt \Big)^{1/2}
              \leq \frac{C}{|\theta-\varphi|}, \quad \theta, \varphi \in (0,\pi), \ \theta \neq \varphi.
    \end{align}
    Here $C>0$ does not depend on $N \in \N$.

    Let $f \in L^2(0,\pi)$ and $\theta \notin \supp(f)$. If $h \in L^2((1/N,N), dt/t)$, \eqref{eq:G7} allows us to write
    \begin{align*}
        & \int_{1/N}^N h(t) [T_{\alpha, \beta}^{\gamma, N}(f)(\theta)](t) \frac{dt}{t}
            =\int_0^\pi f(\varphi) \int_{1/N}^N h(t) [K_{\alpha, \beta}^{\gamma, N}(\theta,\varphi)](t) \frac{dt}{t} d\varphi \\
        & \qquad   = \int_0^\pi f(\varphi) \int_{1/N}^N h(t) t^\gamma \partial_t^\gamma P_t^{\alpha, \beta}(\theta,\varphi) \frac{dt}{t} d\varphi
          = \int_{1/N}^N h(t) \int_0^\pi t^\gamma \partial_t^\gamma P_t^{\alpha, \beta}(\theta,\varphi) f(\varphi)  d\varphi \frac{dt}{t}.
    \end{align*}
    Thus, we obtain
    $$[T_{\alpha, \beta}^{\gamma, N}(f)(\theta)](t)
        = t^\gamma \partial_t^\gamma P_t^{\alpha, \beta}(f)(\theta), \quad \text{a.e. } t \in (1/N,N).$$
    We are going to show, for every $N \in \N$ and $(\theta,\varphi) \in (0,\pi)^2 \setminus \mathcal{D}$, $\theta \neq \varphi$,
    \begin{equation}\label{eq:G8}
        \Big\| \partial_\theta \Big( t^\gamma \partial_t^\gamma P_t^{\alpha, \beta}(\theta,\varphi) \Big) \Big\|_{L^2((1/N,N), dt/t)}
            + \Big\| \partial_\varphi \Big( t^\gamma \partial_t^\gamma P_t^{\alpha, \beta}(\theta,\varphi) \Big) \Big\|_{L^2((1/N,N), dt/t)}
            \leq \frac{C}{|\theta-\varphi|^2},
    \end{equation}
    for a certain $C>0$ which does not depend on $N$ and the domain $\mathcal{D}$ is as in Figure~\ref{fig:regions}.

    To simplify we call
    $$\Phi_{\alpha,\beta}(t,z)
        = \frac{\sinh \frac{t}{2}}{(\cosh \frac{t}{2}-1+z)^{\alpha+\beta+2}}, \quad t,z>0,$$
    to one of the terms appearing in \eqref{eq:15.1}.
    According to \cite[Lemma 4.8]{NoSj2} we have that, for every $m \in \N$,
    \begin{equation}\label{eq:G9}
        \Big| \partial_t^m \Phi_{\alpha,\beta}(t,z) \Big|
            \leq C \left\{\begin{array}{ll}
                        (\cosh \frac{t}{2}-1+z)^{-\alpha-\beta-(m+3)/2}, & t \leq 1, \ z>0 \\
                        (\cosh \frac{t}{2}-1+z)^{-\alpha-\beta-1}, & t > 1, \ z>0,
                   \end{array} \right.
    \end{equation}
    and
    \begin{align}\label{eq:G10}
        & \Big| \partial_\theta \partial_t^m \Phi_{\alpha,\beta}(t,q(\theta,\varphi,u,v)) \Big|
            + \Big| \partial_\varphi \partial_t^m \Phi_{\alpha,\beta}(t,q(\theta,\varphi,u,v)) \Big|  \nonumber \\
        & \qquad \leq C \left\{\begin{array}{ll}
                        (\cosh \frac{t}{2}-1+q(\theta,\varphi,u,v))^{-\alpha-\beta-(m+4)/2}, & t \leq 1, \ \theta,\varphi \in (0,\pi), \ -1<u,v<1 \\
                        (\cosh \frac{t}{2}-1+q(\theta,\varphi,u,v))^{-\alpha-\beta-3/2}, & t > 1, \ \theta,\varphi \in (0,\pi), \ -1<u,v<1.
                   \end{array} \right.
    \end{align}
    Let $m \in \N$. By using \eqref{eq:G9} and \cite[Lemma 4.4]{NoSj2} we get
    \begin{align*}
        & \int_{-1}^1 \int_{-1}^1 \Big| \partial_t^m \Phi_{\alpha,\beta}(t,q(\theta,\varphi,u,v))\Big| d\Pi_\alpha(u) d\Pi_\beta(v) \\
        & \qquad \leq C \left\{\begin{array}{ll}
                            \displaystyle \int_{-1}^1 \int_{-1}^1 \frac{d\Pi_\alpha(u) d\Pi_\beta(v)}{(\cosh \frac{t}{2}-1+q(\theta,\varphi,u,v))^{\alpha+\beta+(m+3)/2}}, & t \leq 1 \\
                            \quad \\
                            \displaystyle \int_{-1}^1 \int_{-1}^1 \frac{d\Pi_\alpha(u) d\Pi_\beta(v)}{(\cosh \frac{t}{2}-1+q(\theta,\varphi,u,v))^{\alpha+\beta+1}}, & t > 1 \\
                        \end{array} \right. \\
        & \qquad \leq \frac{C}{(\cosh \frac{t}{2}-1)^{\alpha+\beta+1}}, \quad t>0 \text{ and } \theta, \varphi \in (0,\pi).
    \end{align*}
    Thus, from \eqref{eq:15.1} we can write for each $\theta,\varphi \in (0,\pi)$ and $t>0$,
    \begin{align*}
        & \partial_t^m P_t^{\alpha, \beta}(\theta,\varphi) \\
        & \qquad = C_{\alpha,\beta} \Big(\sin \frac{\theta}{2} \sin \frac{\varphi}{2}\Big)^{\alpha+1/2} \Big(\cos \frac{\theta}{2} \cos \frac{\varphi}{2}\Big)^{\beta+1/2}
            \int_{-1}^1 \int_{-1}^1 \partial_t^m \Phi_{\alpha,\beta}(t,q(\theta,\varphi,u,v)) d\Pi_\alpha(u) d\Pi_\beta(v).
    \end{align*}
    Assume that $m \in \N$ is such that $m-1 \leq \gamma < m$. From \eqref{eq:G10} and
    \cite[trigonometric identities in p. 738]{NoSj2} we deduce, for every $\theta,\varphi \in (0,\pi)$ and $t>0$,
    \begin{align*}
        & \int_0^\infty s^{m-\gamma-1}\int_{-1}^1 \int_{-1}^1
            \Big| \partial_\theta \partial_t^m \Phi_{\alpha,\beta}(t+s,q(\theta,\varphi,u,v)) \Big| d\Pi_\alpha(u) d\Pi_\beta(v) ds \\
        & \qquad \leq C \Big\{ \int_0^{\max\{0,1-t\}} \frac{s^{m-\gamma-1}}{(\cosh \frac{t+s}{2}-1+2 \sin^2 \frac{\theta-\varphi}{4})^{\alpha+\beta+(m+4)/2}} ds \\
        & \qquad \qquad + \int_{\max\{0,1-t\}}^1 \frac{s^{m-\gamma-1}}{(\cosh \frac{t+s}{2}-1+2 \sin^2 \frac{\theta-\varphi}{4})^{\alpha+\beta+3/2}} ds \\
        & \qquad \qquad + \int_1^\infty s^{m-\gamma-1} e^{-c(\alpha+\beta+3/2)(t+s)} ds \Big\}
        < \infty.
    \end{align*}
    Hence, we can write for $\theta,\varphi \in (0,\pi)$ and $t>0$,
    \begin{align*}
        & t^\gamma \partial_\theta  \partial_t^\gamma P_t^{\alpha, \beta}(\theta,\varphi)
            = C_{\alpha,\beta} \Big(\sin \frac{\theta}{2} \sin \frac{\varphi}{2}\Big)^{\alpha+1/2} \Big(\cos \frac{\theta}{2} \cos \frac{\varphi}{2}\Big)^{\beta+1/2} \frac{e^{-i(m-\gamma)\pi}}{\Gamma(m-\gamma)} t^\gamma \\
        & \qquad \times \Big[ \int_0^\infty s^{m-\gamma-1} \int_{-1}^1 \int_{-1}^1 \partial_\theta \partial_t^m \Phi_{\alpha,\beta}(t+s,q(\theta,\varphi,u,v)) d\Pi_\alpha(u) d\Pi_\beta(v)ds \\
        & \qquad  + \Big( \frac{2\alpha+1}{4} \frac{\cos \frac{\theta}{2}}{\sin \frac{\theta}{2}} -\frac{2\beta+1}{4} \frac{\sin \frac{\theta}{2}}{\cos \frac{\theta}{2}} \Big)
                        \int_0^\infty s^{m-\gamma-1} \int_{-1}^1 \int_{-1}^1 \partial_t^m \Phi_{\alpha,\beta}(t+s,q(\theta,\varphi,u,v)) d\Pi_\alpha(u) d\Pi_\beta(v)ds\Big].
    \end{align*}
    By proceeding as in \cite[pp. 747-748]{NoSj2} (see also the proof of Proposition~\ref{propT2}),
    \eqref{eq:G10} and Minkowski's inequality leads to
    \begin{align}\label{eq:G11}
        & \Big\| \Big(\sin \frac{\theta}{2} \sin \frac{\varphi}{2}\Big)^{\alpha+1/2} \Big(\cos \frac{\theta}{2} \cos \frac{\varphi}{2}\Big)^{\beta+1/2} t^\gamma \nonumber \\
        & \qquad \qquad \times \int_0^\infty s^{m-\gamma-1} \int_{-1}^1 \int_{-1}^1 \partial_\theta \partial_t^m \Phi_{\alpha,\beta}(t+s,q(\theta,\varphi,u,v)) d\Pi_\alpha(u) d\Pi_\beta(v)ds    \Big\|_{L^2((0,\infty),dt/t)} \nonumber \\
        & \qquad \leq \Big(\sin \frac{\theta}{2} \sin \frac{\varphi}{2}\Big)^{\alpha+1/2} \Big(\cos \frac{\theta}{2} \cos \frac{\varphi}{2}\Big)^{\beta+1/2} \int_0^\infty s^{m-\gamma-1} \int_{-1}^1 \int_{-1}^1 \nonumber \\
        & \qquad \qquad \times \Big\|  t^\gamma  \partial_\theta \partial_t^m \Phi_{\alpha,\beta}(t+s,q(\theta,\varphi,u,v)) \Big\|_{L^2((0,\infty),dt/t)} d\Pi_\alpha(u) d\Pi_\beta(v)ds \nonumber\\
        & \qquad \leq C \Big(\sin \frac{\theta}{2} \sin \frac{\varphi}{2}\Big)^{\alpha+1/2} \Big(\cos \frac{\theta}{2} \cos \frac{\varphi}{2}\Big)^{\beta+1/2} \int_{-1}^1 \int_{-1}^1
            \frac{d\Pi_\alpha(u) d\Pi_\beta(v)ds}{q(\theta,\varphi,u,v)^{\alpha+\beta+2}}  \nonumber\\
        & \qquad \leq \frac{C}{|\theta - \varphi|^2}, \quad \theta,\varphi \in (0,\pi), \ \theta \neq \varphi.
    \end{align}
    In a similar way, by using \eqref{eq:G9} we obtain
        \begin{align}\label{eq:G12}
        & \Big\| \Big(\sin \frac{\theta}{2} \sin \frac{\varphi}{2}\Big)^{\alpha+1/2} \Big(\cos \frac{\theta}{2} \cos \frac{\varphi}{2}\Big)^{\beta+1/2} t^\gamma \nonumber \\
        & \qquad \qquad \times \int_0^\infty s^{m-\gamma-1} \int_{-1}^1 \int_{-1}^1 \partial_t^m \Phi_{\alpha,\beta}(t+s,q(\theta,\varphi,u,v)) d\Pi_\alpha(u) d\Pi_\beta(v)ds    \Big\|_{L^2((0,\infty),dt/t)} \nonumber \\
        & \qquad \leq C \Big(\sin \frac{\theta}{2} \sin \frac{\varphi}{2}\Big)^{\alpha+1/2} \Big(\cos \frac{\theta}{2} \cos \frac{\varphi}{2}\Big)^{\beta+1/2} \int_{-1}^1 \int_{-1}^1
            \frac{d\Pi_\alpha(u) d\Pi_\beta(v)ds}{q(\theta,\varphi,u,v)^{\alpha+\beta+3/2}}  \nonumber\\
        & \qquad \leq \frac{C}{|\theta - \varphi|} \leq \frac{C}{|\theta - \varphi|^2}, \quad \theta,\varphi \in (0,\pi), \ \theta \neq \varphi.
    \end{align}
    Combining \eqref{eq:G11} and \eqref{eq:G12} with \eqref{eq:17.1}, \eqref{eq:17.2}, \eqref{eq:17.3} and \eqref{eq:17.4},
    we deduce that
    $$\Big\| \partial_\theta K_{\alpha,\beta}^\gamma(\theta,\varphi) \Big\|_{L^2((0,\infty),dt/t)}
        \leq \frac{C}{|\theta - \varphi|^2}, \quad (\theta, \varphi) \in (0,\pi)^2 \setminus \mathcal{D}.$$
    The same procedure allows us to prove that
    $$\Big\| \partial_\varphi K_{\alpha,\beta}^\gamma(\theta,\varphi) \Big\|_{L^2((0,\infty),dt/t)}
        \leq \frac{C}{|\theta - \varphi|^2}, \quad (\theta, \varphi) \in (0,\pi)^2 \setminus \mathcal{D}.$$
    Thus, \eqref{eq:G8} is established.

    By using now the local Calder\'on-Zygmund theory for singular integrals (see \cite{CNS}) in the $L^2((1/N,N),dt/t)$-setting
    and by taking into account Lemma~\ref{lemT2}, we conclude that, for every $1<p<\infty$ and $w \in A_p(0,\pi)$,
    the operator $T_{\alpha,\beta}^{\gamma,N}$ can be extended from $L^2(0,\pi) \cap L^p_w(0,\pi)$ to $L^p_w(0,\pi)$
    as a bounded operator $\widetilde{T}_{\alpha,\beta}^{\gamma,N}$ from $L^p_w(0,\pi)$ into
    $L^p_w\big((0,\pi); L^2((1/N,N),dt/t) \big)$, and there exists $C>0$, which does not depend on $N$, such that
    \begin{equation}\label{eq:G13}
        \Big\| \widetilde{T}_{\alpha,\beta}^{\gamma,N}(f) \Big\|_{L^p_w\big((0,\pi); L^2((1/N,N),dt/t) \big)}
            \leq C \|f\|_{L^p_w(0,\pi)}, \quad f \in L^p_w(0,\pi).
    \end{equation}
    Let $f \in L^p_w(0,\pi)$ where $1<p<\infty$ and $w \in A_p(0,\pi)$. We take a sequence
    $(f_n)_{n \in \N} \subseteq L^p_w(0,\pi) \cap L^2(0,\pi)$ such that
    $$f_n
        \longrightarrow f, \quad \text{as } n \to \infty, \text{ in } L^p_w(0,\pi).$$
    As in \eqref{eq:G5} we obtain that
    $$\mathcal{G}_{\alpha,\beta}^{\gamma,N}(f-f_n)(\theta)
        \leq C \|f-f_n\|_{L^p_w(0,\pi)}, \quad n \in \N \text{ and } \theta \in (0,\pi).$$
    Hence,
    $$\mathcal{G}_{\alpha,\beta}^{\gamma,N}(f_n)(\theta)
        \longrightarrow \mathcal{G}_{\alpha,\beta}^{\gamma,N}(f)(\theta), \quad \text{as } n \to \infty \text{ for every } \theta \in (0,\pi).$$
    On the other hand,
    $$\widetilde{T}_{\alpha,\beta}^{\gamma,N}(f)
        = \lim_{n \to \infty} T_{\alpha,\beta}^{\gamma,N}(f_n), \quad \text{in } L^p_w\big((0,\pi); L^2((1/N,N),dt/t) \big). $$
    Then, there exists a monotone function $\phi : \N \longrightarrow \N$ such that
    $$T_{\alpha,\beta}^{\gamma,N}(f_{\phi(n)})(\theta)
        \longrightarrow \widetilde{T}_{\alpha,\beta}^{\gamma,N}(f)(\theta),
        \quad \text{as } n \to \infty, \text{ in } L^2((1/N,N),dt/t),$$
    for almost every $\theta \in (0,\pi)$. This implies that
    $$\mathcal{G}_{\alpha,\beta}^{\gamma,N}(f_{\phi(n)})(\theta)
        \longrightarrow \Big\| \widetilde{T}_{\alpha,\beta}^{\gamma,N}(f)(\theta) \Big\|_{L^2((1/N,N),dt/t)},
        \quad \text{as } n \to \infty,$$
    for almost every $\theta \in (0,\pi)$. We conclude that
    $$\mathcal{G}_{\alpha,\beta}^{\gamma,N}(f)(\theta)
        = \Big\| \widetilde{T}_{\alpha,\beta}^{\gamma,N}(f)(\theta) \Big\|_{L^2((1/N,N),dt/t)}, \quad \text{a.e. } \theta \in (0,\pi),$$
    and from \eqref{eq:G13} we deduce \eqref{eq:G1}.

    Thus the proof of this proposition is completed.
\end{proof}

By using \cite[Theorem 1.3]{CFMP} from Proposition~\ref{PropG1} we infer the following.

\begin{Cor}\label{Cor:G2}
    Let $\alpha,\beta \geq -1/2$ and $\gamma>0$. Suppose that $p \in \mathcal{B}(0,\pi)$.
    Then, the fractional square function $g_{\alpha,\beta}^\gamma$
    defines a bounded (quasi-linear) operator from $L^{p(\cdot)}(0,\pi)$ into itself.
\end{Cor}

Also Proposition~\ref{PropG1} and the polarization formula \eqref{eq:G2} allow us to obtain the converse inequality for
$g_{\alpha,\beta}^\gamma$.

\begin{Cor}\label{Cor:G3}
    Let $\alpha,\beta \geq -1/2$ and $\gamma>0$.
    \begin{itemize}
        \item[$(a)$] If $1<p<\infty$ and $w \in A_p(0,\pi)$ then, for a certain $C>0$,
        $$\|f\|_{L^p_w(0,\pi)}
            \leq C \| g_{\alpha,\beta}^\gamma(f) \|_{L^p_w(0,\pi)}, \quad f \in L^p_w(0,\pi).$$
        \item[$(b)$] If $p \in \mathcal{B}(0,\pi)$, then there exits $C>0$ such that
        $$\|f\|_{L^{p(\cdot)}(0,\pi)}
            \leq C \| g_{\alpha,\beta}^\gamma(f) \|_{L^{p(\cdot)}(0,\pi)}, \quad f \in L^{p(\cdot)}(0,\pi).$$
    \end{itemize}
\end{Cor}

\begin{proof}
    We are going to prove $(b)$, $(a)$ can be deduced in a similar way.

    For every $f\in L^{p(\cdot)}(0,\pi)$ and $g\in L^{p'(\cdot)}(0,\pi)$, we consider the bilinear operators
    $$T(f,g)
        = \int_0^\pi f(\theta) \overline{g}(\theta) d\theta, $$
    and
    $$L(f,g)
        = \frac{2^{2\gamma}}{\Gamma(2\gamma)} \int_0^\pi \int_0^\infty t^\gamma \partial_t^\gamma P_t^{\alpha,\beta}f(\theta) t^\gamma \partial_t^\gamma P_t^{\alpha,\beta}(\overline{g})(\theta) \frac{dt}{t} d\theta.$$
    By using H\"older's inequality in the variable exponent setting
    (see \cite[Lemma 3.2.20]{DHHR}) we can see that $T$ and $L$ are bounded from
    $L^{p(\cdot)}(0,\pi) \times L^{p'(\cdot)}(0,\pi)$ into $\C$. Since $S_{\alpha,\beta}$ is a dense subspace of
    $L^{p(\cdot)}(0,\pi)$ and $L^{p'(\cdot)}(0,\pi)$ (Proposition~\ref{propS1}), equality \eqref{eq:G2}
    holds for every $f \in L^{p(\cdot)}(0,\pi)$ and $g \in L^{p'(\cdot)}(0,\pi)$.

    Let $f \in L^{p(\cdot)}(0,\pi)$. According to the norm conjugate formula (\cite[Corollary 3.2.14]{DHHR}), by Proposition~\ref{PropG1}
    we can write
    \begin{align*}
        \|f\|_{L^{p(\cdot)}(0,\pi)}
            & \leq 2 \sup_{\substack{g \in L^{p'(\cdot)}(0,\pi) \\ \|g\|_{L^{p'(\cdot)}(0,\pi)} \leq 1 }}
                \Big| \int_0^\pi f(\theta) \overline{g}(\theta) d\theta \Big| \\
            & \leq C \sup_{\substack{g \in L^{p'(\cdot)}(0,\pi) \\ \|g\|_{L^{p'(\cdot)}(0,\pi)} \leq 1 }}
                \Big| \int_0^\pi \int_0^\infty t^\gamma \partial_t^\gamma P_t^{\alpha,\beta}f(\theta) t^\gamma \partial_t^\gamma P_t^{\alpha,\beta}(\overline{g})(\theta) \frac{dt}{t} d\theta \Big| \\
            & \leq C \sup_{\substack{g \in L^{p'(\cdot)}(0,\pi) \\ \|g\|_{L^{p'(\cdot)}(0,\pi)} \leq 1 }}
                \int_0^\pi g_{\alpha,\beta}^\gamma(f)(\theta) g_{\alpha,\beta}^\gamma(\overline{g})(\theta) d\theta  \\
            & \leq C \sup_{\substack{g \in L^{p'(\cdot)}(0,\pi) \\ \|g\|_{L^{p'(\cdot)}(0,\pi)} \leq 1 }}
                \| g_{\alpha,\beta}^\gamma(f)\|_{L^{p(\cdot)}(0,\pi)} \  \|g_{\alpha,\beta}^\gamma(\overline{g}) \|_{L^{p'(\cdot)}(0,\pi)}
             \leq C \| g_{\alpha,\beta}^\gamma(f)\|_{L^{p(\cdot)}(0,\pi)}.
    \end{align*}
\end{proof}

\begin{Rem}
    Note that Proposition~\ref{PropG1} together with Corollaries~\ref{Cor:G2} and \ref{Cor:G3} tell us that the new norms
    $||| \cdot |||_{L^{p}_w(0,\pi)}$ and $||| \cdot |||_{L^{p(\cdot)}(0,\pi)}$ defined by
    \begin{align*}
        ||| f |||_{L^p_w(0,\pi)}
            & = \| g_{\alpha,\beta}^\gamma(f) \|_{L^p_w(0,\pi)}, \quad f \in L^p_w(0,\pi),\\
        ||| f |||_{L^{p(\cdot)}(0,\pi)}
            & = \| g_{\alpha,\beta}^\gamma(f) \|_{L^{p(\cdot)}(0,\pi)}, \quad f \in L^{p(\cdot)}(0,\pi),
    \end{align*}
    are equivalent to $\| \cdot \|_{L^p_w(0,\pi)}$ on $L^p_w(0,\pi)$ and to $\| \cdot \|_{L^{p(\cdot)}(0,\pi)}$ on $L^{p(\cdot)}(0,\pi)$,
    respectively, provided that the specified conditions are satisfied.
\end{Rem}

\begin{proof}[Proof of Theorem~\ref{Th2}]
    We first establish that $H_{\alpha,\beta}^{\gamma/2,p(\cdot)}(0,\pi) \subseteq T_{\alpha, \beta}^{\gamma,k,p(\cdot)}(0,\pi)$.
    Assume that $f,g \in S_{\alpha,\beta}$ are such that $f=\mathcal{L}^{-\gamma}_{\alpha,\beta}g$.
    We can write
    \begin{align*}\label{eq:G14}
        \partial_t^k P_t^{\alpha,\beta} \big( \mathcal{L}_{\alpha,\beta}^{-\gamma/2} g \big)
            & = (-1)^k \sum_{n=0}^\infty \frac{e^{-t \sqrt{\lambda_n^{\alpha,\beta}}}}{(\lambda_n^{\alpha,\beta})^{(\gamma-k)/2}} c_n^{\alpha,\beta}(g) \phi_n^{\alpha,\beta}
            = e^{i \pi \gamma} \partial_t^{k-\gamma}P_t^{\alpha,\beta}g, \quad t>0,
    \end{align*}
    because $\partial_t^\delta e^{-at}=e^{i \pi \delta} a^\delta e^{-at}$, $\delta,a,t>0$.
    Hence, we get
    \begin{equation}\label{eq:G15}
        g_{\alpha,\beta}^{\gamma,k}\big( \mathcal{L}_{\alpha,\beta}^{-\gamma/2} g \big)
            = g_{\alpha,\beta}^{k-\gamma}(g).
    \end{equation}
    From \eqref{eq:G15} and Corollaries~\ref{Cor:G2} and \ref{Cor:G3} we deduce that, for every $f \in S_{\alpha,\beta}$,
    \begin{equation}\label{eq:G16}
        \frac{1}{C} \|f\|_{H_{\alpha,\beta}^{\gamma/2,p(\cdot)}(0,\pi)}
            \leq \| g_{\alpha,\beta}^{\gamma,k}(f)\|_{L^{p(\cdot)}(0,\pi)}
            \leq C \|f\|_{H_{\alpha,\beta}^{\gamma/2,p(\cdot)}(0,\pi)},
    \end{equation}
    for a certain $C>0$. Since $S_{\alpha,\beta}$ is a dense subspace of $H_{\alpha,\beta}^{\gamma/2,p(\cdot)}(0,\pi)$,
    $g_{\alpha,\beta}^{\gamma,k}$ can be extended to $H_{\alpha,\beta}^{\gamma/2,p(\cdot)}(0,\pi)$ as a bounded operator
    $\widetilde{g}_{\alpha,\beta}^{\gamma,k}$ from $H_{\alpha,\beta}^{\gamma/2,p(\cdot)}(0,\pi)$ into $L^{p(\cdot)}(0,\pi)$.
    Moreover, \eqref{eq:G16} holds for every $f \in H_{\alpha,\beta}^{\gamma/2,p(\cdot)}(0,\pi)$ when $g_{\alpha,\beta}^{\gamma,k}$
    is replaced by $\widetilde{g}_{\alpha,\beta}^{\gamma,k}$.

    We are going to see that $\widetilde{g}_{\alpha,\beta}^{\gamma,k} = g_{\alpha,\beta}^{\gamma,k}$. For every $N \in \N$, we define
    $$\mathcal{G}_{\alpha,\beta}^{\gamma,k,N}(f)(\theta)
        = \Big( \int_{1/N}^N \big| t^{k-\gamma} \partial_t^k P_t^{\alpha,\beta}f(\theta)\big|^2 \frac{dt}{t}\Big)^{1/2}, \quad \theta \in (0,\pi).$$
    Let $N \in \N$. From \eqref{eq:G16} it follows that $\mathcal{G}_{\alpha,\beta}^{\gamma,k,N}$ can be extended to
    $H_{\alpha,\beta}^{\gamma/2,p(\cdot)}(0,\pi)$ as a bounded operator $\widetilde{\mathcal{G}}_{\alpha,\beta}^{\gamma,k,N}$
    from $H_{\alpha,\beta}^{\gamma/2,p(\cdot)}(0,\pi)$ into $L^{p(\cdot)}(0,\pi)$ and
    $$\|\widetilde{\mathcal{G}}_{\alpha,\beta}^{\gamma,k,N}(f)\|_{L^{p(\cdot)}(0,\pi)}
            \leq C \|f\|_{H_{\alpha,\beta}^{\gamma/2,p(\cdot)}(0,\pi)}, \quad f \in H_{\alpha,\beta}^{\gamma/2,p(\cdot)}(0,\pi).$$
    Note that $C$ does not depend on $N$. Let $f \in H_{\alpha,\beta}^{\gamma/2,p(\cdot)}(0,\pi)$. We choose a sequence
    $(f_n)_{n \in \N} \subseteq S_{\alpha,\beta}$ such that
    $$f_n
        \longrightarrow f, \quad \text{as } n \to \infty, \text{ in } H_{\alpha,\beta}^{\gamma/2,p(\cdot)}(0,\pi).$$
    Then,
    $$\mathcal{G}_{\alpha,\beta}^{\gamma,k,N}(f_n)
        \longrightarrow \widetilde{\mathcal{G}}_{\alpha,\beta}^{\gamma,k,N}(f), \quad \text{as } n \to \infty, \text{ in } L^{p(\cdot)}(0,\pi).$$
    Since, $L^{p(\cdot)}(0,\pi) \subseteq L^{p_-}(0,\pi)$, there exists a monotone function $\phi : \N \longrightarrow \N$ such that
    $$\mathcal{G}_{\alpha,\beta}^{\gamma,k,N}(f_{\phi(n)})(\theta)
        \longrightarrow \widetilde{\mathcal{G}}_{\alpha,\beta}^{\gamma,k,N}(f)(\theta),
        \quad \text{as } n \to \infty,
        \quad \text{a.e. } \theta \in (0,\pi). $$
    By proceeding as in \eqref{eq:G5} we deduce that
    $$\mathcal{G}_{\alpha,\beta}^{\gamma,k,N}(f_{\phi(n)})(\theta)
        \longrightarrow \mathcal{G}_{\alpha,\beta}^{\gamma,k,N}(f)(\theta),
        \quad \text{as } n \to \infty,
        \quad \theta \in (0,\pi).$$
    Then, $\widetilde{\mathcal{G}}_{\alpha,\beta}^{\gamma,k,N} = \mathcal{G}_{\alpha,\beta}^{\gamma,k,N}$ and
    $$\| \mathcal{G}_{\alpha,\beta}^{\gamma,k,N}(f)\|_{L^{p(\cdot)}(0,\pi)}
            \leq C \|f\|_{H_{\alpha,\beta}^{\gamma/2,p(\cdot)}(0,\pi)}.$$
    Since
    $$\lim_{N \to \infty} \mathcal{G}_{\alpha,\beta}^{\gamma,k,N}(f)(\theta)
        = g_{\alpha,\beta}^{\gamma,k}(f)(\theta), \quad \theta \in (0,\pi),$$
    Fatou's Lemma in variable exponent $L^{p(\cdot)}$-spaces (see \cite[p. 77]{DHHR}) leads to
    \begin{equation}\label{eq:G17}
        \|g_{\alpha,\beta}^{\gamma,k}(f)\|_{L^{p(\cdot)}(0,\pi)}
            \leq C \|f\|_{H_{\alpha,\beta}^{\gamma/2,p(\cdot)}(0,\pi)}.
    \end{equation}
    From \eqref{eq:G16} we also deduce now that
    \begin{equation}\label{eq:G18}
        \|f\|_{H_{\alpha,\beta}^{\gamma/2,p(\cdot)}(0,\pi)}
            \leq C \| g_{\alpha,\beta}^{\gamma,k}(f)\|_{L^{p(\cdot)}(0,\pi)}, \quad f \in H_{\alpha,\beta}^{\gamma/2,p(\cdot)}(0,\pi).
    \end{equation}
    By \eqref{eq:G17} it follows that $H_{\alpha,\beta}^{\gamma/2,p(\cdot)}(0,\pi)$ is contained in
    $T_{\alpha,\beta}^{\gamma,k,p(\cdot)}(0,\pi)$ and by Proposition~\ref{propM3}
    $$\|f\|_{T_{\alpha,\beta}^{\gamma,k,p(\cdot)}(0,\pi)}
        \leq C \|f\|_{H_{\alpha,\beta}^{\gamma/2,p(\cdot)}(0,\pi)}, \quad f \in H_{\alpha,\beta}^{\gamma/2,p(\cdot)}(0,\pi).$$

    Suppose now that $f \in T_{\alpha,\beta}^{\gamma,k,p(\cdot)}(0,\pi)$.
    In order to show that $f \in H_{\alpha,\beta}^{\gamma/2,p(\cdot)}(0,\pi)$ we can follow the
    procedure developed in the proof of \cite[Proposition 4.1]{BFRTT2}. Indeed, that method works
    because the following properties hold:
    \begin{itemize}
        \item[$(i)$] There exists $C>0$ such that, for every $n \in \N$,
        $$\| \phi_n^{\alpha,\beta}\|_{L^{p(\cdot)}(0,\pi)}
            \leq C (n+1)^{\alpha + \beta + 5/2}.$$
        Indeed, according to \cite[Theorem 3.3.11]{DHHR}, $L^{p_+}(0,\pi)$ is continuously contained in
        $L^{p(\cdot)}(0,\pi)$. Then, from \cite[(3)]{NoSj2} it follows that
        $$\| \phi_n^{\alpha,\beta}\|_{L^{p(\cdot)}(0,\pi)}
            \leq C \| \phi_n^{\alpha,\beta}\|_{L^{p_+}(0,\pi)}
            \leq C(n+1)^{\alpha + \beta + 5/2}, \quad n \in \N. $$
        Assume that $h \in L^{p(\cdot)}(0,\pi)$. H\"older's inequality (\cite[Lemma 3.2.20]{DHHR}) implies that
        $$|c_n^{\alpha,\beta}(h)|
            \leq C (n+1)^{\alpha + \beta + 5/2} \|h\|_{L^{p(\cdot)}(0,\pi)}, \quad n \in \N. $$

        \item[$(ii)$] For every $\delta>0$, we define $f_\delta=P_\delta^{\alpha,\beta}(f)$ and
        $$F_\delta
            = \sum_{n=0}^\infty (\lambda_n^{\alpha,\beta})^{\gamma/2} e^{-\delta \sqrt{\lambda_n^{\alpha,\beta}}}
                c_n^{\alpha,\beta}(f_\delta) \phi_n^{\alpha,\beta}.$$
        Property $(i)$ implies that $F_\delta \in L^{p(\cdot)}(0,\pi)$ and
        $f_\delta=\mathcal{L}_{\alpha,\beta}^{-\gamma/2} F_\delta \in H_{\alpha,\beta}^{\gamma/2,p(\cdot)}(0,\pi)$, $\delta>0$.
        We choose $\ell \in \N$ such that $2(\ell-\gamma)>1$ and $\ell >k$. \eqref{eq:G18} allows us to write
        $$\|F_\delta\|_{L^{p(\cdot)}(0,\pi)}
            = \|f_\delta\|_{H_{\alpha,\beta}^{\gamma/2,p(\cdot)}(0,\pi)}
            \leq C \| g_{\alpha,\beta}^{\gamma,\ell}(f_\delta) \|_{L^{p(\cdot)}(0,\pi)}, \quad \delta>0.$$

        \item[$(iii)$] As in \cite[Proposition 2.6]{BFRTT2} we can prove that
        $$\| g_{\alpha,\beta}^{\gamma,\ell}(f) \|_{L^{p(\cdot)}(0,\pi)}
            \leq C \| g_{\alpha,\beta}^{\gamma,k}(f) \|_{L^{p(\cdot)}(0,\pi)}.$$
        Moreover, straightforward manipulations lead to
        $$g_{\alpha,\beta}^{\gamma,\ell}(f_\delta)(\theta)
            \leq g_{\alpha,\beta}^{\gamma,\ell}(f)(\theta), \quad \theta \in (0,\pi), \ \delta>0,$$
        because $2(\ell-\gamma)>1$. Then, we obtain
        $$\|F_\delta\|_{L^{p(\cdot)}(0,\pi)}
            \leq C \| g_{\alpha,\beta}^{\gamma,k}(f) \|_{L^{p(\cdot)}(0,\pi)}, \quad \delta>0.$$

        \item[$(iv)$] By using Banach-Alaoglu's Theorem, Proposition~\ref{propM3} and \cite[Theorem 3.2.13]{DHHR}
        we conclude that $f=\mathcal{L}_{\alpha,\beta}^{-\gamma/2}F$, for a certain $F \in L^{p(\cdot)}(0,\pi)$
        such that
        $$\|F\|_{L^{p(\cdot)}(0,\pi)}
            \leq C \| g_{\alpha,\beta}^{\gamma,k}(f) \|_{L^{p(\cdot)}(0,\pi)}.$$
    \end{itemize}

    Thus, we prove that $f \in H_{\alpha,\beta}^{\gamma/2,p(\cdot)}(0,\pi)$ and
    $$\|f\|_{H_{\alpha,\beta}^{\gamma/2,p(\cdot)}(0,\pi)}
        \leq C \|f\|_{T_{\alpha,\beta}^{\gamma,k,p(\cdot)}(0,\pi)}.$$
\end{proof}

\section{Proof of Theorem~\ref{Th3}} \label{sec:6}

In order to establish this theorem we use the ideas developed in the proof of \cite[Proposition~4.3]{NPW}.
First of all, we introduce some spectral multipliers of H\"ormander type, associated with the Jacobi operator.

\begin{Lem}\label{Lem:6.1}
	Let $\gamma>0$, $1<p<\infty$, $w \in A_p(0,\pi)$ and  $\alpha,\beta \geq -1/2$ such that $\alpha + \beta \neq -1$. We consider, for each $t>0$,
	the functions
	\begin{itemize}
		\item $\displaystyle m_{\varepsilon}^{\ell}(t)=\sum_{j=0}^\ell {{\varepsilon_j2^{j\gamma}}\over{(t+1)^{\gamma}}}\mathfrak{a}\left({t\over{2^{j-1}}}\right)$, \quad
	     $\ell\in\Bbb N$ and $\varepsilon=(\varepsilon_j)_{j=0}^\ell \in\{-1,1\}^{\ell+1}$.
		
		\item $\displaystyle M(t)=\left({{t+1}\over t}\right)^{\gamma}\phi(t)$, \quad
		where $\phi\in C^{\infty}(0,\infty)$ is such that $\phi(t)=0$, $0<t<\lambda_0^{\alpha,\beta}/2$; and
		$\phi(t)=1$, $t \geq \lambda_0^{\alpha,\beta}$.
	\end{itemize}
	Then, the spectral multipliers $m_{\varepsilon}^{\ell}(\mathcal{L}_{\alpha,\beta})$ and $M(\mathcal{L}_{\alpha,\beta})$ define
	bounded operators in $L^p_w(0,\pi)$. Moreover,
	$$\sup_{\ell, \varepsilon}\| m_{\varepsilon}^{\ell}(\mathcal{L}_{\alpha,\beta}) \|_{ L_{w}^p(0,\pi)\rightarrow  L_{w}^p(0,\pi)}			<\infty.$$
\end{Lem}

\begin{proof}
	By Proposition \ref{propM1}, it is enough to notice that, for every $k\in\Bbb N$, there exists $C>0$ such that
	$$\sup_{t>0}\Big|t^k{{d^k}\over{dt^k}}m_{\varepsilon}^{\ell}(t)\Big|
		\leq C, \quad \ell\in\Bbb N \text{ and } \varepsilon\in\{-1,1\}^{\ell+1}.$$
	 and
	 $$\sup_{t>0} \Big|t^k{{d^k}\over{dt^k}}M(t)\Big|
	 		\leq C.$$
\end{proof}

\begin{proof}[Proof of Theorem \ref{Th3}; the case of $H_{\alpha,\beta}^{\gamma,p(\cdot)}(0,\pi)\subseteq F_{\alpha,\beta}^{\gamma,2,p(\cdot)}(0,\pi)$]
	Let $\varepsilon=(\varepsilon_j)_{j=0}^\ell \in\{-1,1\}^{\ell+1}$ with $\ell\in\Bbb N$. We can write,
   	\begin{align}\label{W0}
     	 	 \sum_{n=0}^{\infty}m_{\varepsilon}^{\ell}({\lambda}_n^{\alpha,\beta})( {\lambda}_n^{\alpha,\beta}+1)^{\gamma}c_n^{\alpha,\beta}(f){\phi}_n^{\alpha,\beta} \nonumber
 		 & =\sum_{n=0}^{\infty}( {\lambda}_n^{\alpha,\beta}+1)^{\gamma}c_n^{\alpha,\beta}(f){\phi}_n^{\alpha,\beta}
		 	\sum_{j=0}^\ell {{\varepsilon_j2^{j\gamma}}\over{( {\lambda}_n^{\alpha,\beta}+1)^{\gamma}}}
				\mathfrak{a}\left({{\lambda}_n^{\alpha,\beta}}\over{2^{j-1}}\right) \nonumber \\
           &   =\sum_{j=0}^\ell \varepsilon_j2^{j\gamma} \sum_{n=0}^{\infty} \mathfrak{a}\left({{\lambda}_n^{\alpha,\beta}}\over{2^{j-1}}\right)
           		c_n^{\alpha,\beta}(f){\phi}_n^{\alpha,\beta} \nonumber\\
            & =\sum_{j=0}^\ell \varepsilon_j2^{j\gamma} {\Phi}_j^{\alpha,\beta}(f), \quad f\in L^p_w(0,\pi).
    \end{align}
	Note that the serie $\displaystyle\sum_{n=0}^{\infty}$ is actually a finite sum.
	From Lemma \ref{Lem:6.1}, it follows that
	\begin{align*}
		\Big \|\sum_{j=0}^\ell \varepsilon_j2^{j\gamma} {\Phi}_j^{\alpha,\beta}(f)\Big\|_{L_{w}^p(0,\pi)}
			& = \Big\|\sum_{n=0}^{\infty}m_{\varepsilon}^{\ell}({\lambda}_n^{\alpha,\beta})M({\lambda}_n^{\alpha,\beta})( {\lambda}_n^{\alpha,\beta})^{\gamma}								c_n^{\alpha,\beta}(f){\phi}_n^{\alpha,\beta} \Big\|_{L_{w}^p(0,\pi)} \\
			& \leq C \Big\|\sum_{n=0}^{\infty}( {\lambda}_n^{\alpha,\beta})^{\gamma}c_n^{\alpha,\beta}(f){\phi}_n^{\alpha,\beta} \Big\|_{L_{w}^p(0,\pi)},
			 \quad f\in L^p_w(0,\pi),
	\end{align*}
	provided that  $\displaystyle\sum_{n=0}^{\infty}( {\lambda}_n^{\alpha,\beta})^{\gamma}c_n^{\alpha,\beta}(f){\phi}_n^{\alpha,\beta}\in L_{w}^p(0,\pi)$.
	Also, we get
	\begin{align}\label{W2}
		\Big\|\sum_{j=0}^\ell \varepsilon_j2^{j\gamma} {\Phi}_j^{\alpha,\beta}(f)\Big\|_{L_{w}^p(0,\pi)}
    			\leq C \Big\|\sum_{\substack{n\in\Bbb N\\ \lambda_n^{\alpha,\beta}\leq 2^\ell} }( {\lambda}_n^{\alpha,\beta})^{\gamma}
				c_n^{\alpha,\beta}(f){\phi}_n^{\alpha,\beta}\Big\|_{L_{w}^p(0,\pi)},  \quad f\in L^p_w(0,\pi).
	\end{align}
	Observe that, the constant $C>0$ does not depend on $\varepsilon$ or $\ell$.

	By using Khintchine's inequality (\cite[Vol. I, p. 213]{Zy}) from (\ref{W2}) we deduce that,
	$$\Big\|\Big(\sum_{j=0}^\ell (2^{j\gamma} |{\Phi}_j^{\alpha,\beta}(f)|)^2\Big)^{1/2}\Big\|_{L_{w}^p(0,\pi)}
    			\leq C \Big\|\sum_{\substack{n\in\Bbb N\\ {\lambda}_n^{\alpha,\beta}\leq 2^\ell}}( {\lambda}_n^{\alpha,\beta})^{\gamma}
					c_n^{\alpha,\beta}(f){\phi}_n^{\alpha,\beta}\Big\|_{L_{w}^p(0,\pi)}, \quad f\in L^p_w(0,\pi),$$
	where $C>0$ does not depend on $\ell$.
	According to \cite[Theorem 1.3]{CFMP}, there exists $C>0$ such that
	$$\Big\|\Big(\sum_{j=0}^\ell (2^{j\gamma} |{\Phi}_j^{\alpha,\beta}(f)|)^2\Big)^{1/2}\Big\|_{L^{p(\cdot)}(0,\pi)}
		\leq C\Big\|\sum_{\substack{n\in\Bbb N\\ {\lambda}_n^{\alpha,\beta}\leq 2^\ell}}( {\lambda}_n^{\alpha,\beta})^{\gamma}
				c_n^{\alpha,\beta}(f){\phi}_n^{\alpha,\beta}\Big\|_{L^{p(\cdot)}(0,\pi)}, \quad f\in L^{p(\cdot)}(0,\pi).$$
	We have taken into account that:
	\begin{enumerate}
		\item[$(a)$] For every $n\in\Bbb N$, the mapping $f \longmapsto c_n^{\alpha,\beta}(f)$ is bounded from $L^{p(\cdot)}(0,\pi)$ into $\Bbb C$.
		\item[$(b)$] For every $j\in\Bbb N$, the mapping $f \longmapsto {\Phi}_j^{\alpha,\beta}(f)$ is bounded from $L^{p(\cdot)}(0,\pi)$ into itself 							(Proposition \ref{propM1}).  Also, we used that $\sqrt{a^2+b^2}\leq a+b$, $a,b\geq 0$.
		\item[$(c)$] $S_{\alpha,\beta}$ is dense in $L^{p(\cdot)}(0,\pi)$ (Proposition \ref{propS1}).
	\end{enumerate}

	Taking $\ell\rightarrow\infty$, Proposition \ref{Prop:M5}  allow us to deduce that
	$$\Big\|\Big(\sum_{j=0}^{\infty} (2^{j\gamma} |{\Phi}_j^{\alpha,\beta}(f)|)^2\Big)^{1/2}\Big\|_{L^{p(\cdot)}(0,\pi)}
		\leq C\|f\|_{H_{\alpha,\beta}^{\gamma,p(\cdot)}(0,\pi)},
		\quad f\in H_{\alpha,\beta}^{\gamma,p(\cdot)}(0,\pi).$$
\end{proof}

Next, we prove the converse inclusion of Theorem \ref{Th3}. As before, we need to study previously some Jacobi spectral multipliers.
It is convenient to introduce the following notation. We define,
$${\Bbb N}_s
	=\{4\ell+s\ : \ \ell\in\Bbb N\}\backslash\{0\}, \quad s=0,1,2,3.$$
Also we consider the function
$$\mathfrak{b}(t)
	=\mathfrak{a}(t/2)+\mathfrak{a}(t)+\mathfrak{a}(2t), \quad t>0.$$
Note that $\supp \mathfrak{b} \subseteq[1/4,4]$ and $\mathfrak{b}(t)=1$, $t\in[1/2,2]$, because $\mathfrak{a}(t)+\mathfrak{a}(2t)=1$, $t\in[1/2,1]$, and
$\supp \mathfrak{a} \subseteq[1/2,2]$.

\begin{Lem}\label{Lem:6.2}
	Let $1<p<\infty$, $w \in A_p(0,\pi)$ and  $\alpha,\beta \geq -1/2$ such that $\alpha + \beta \neq -1$. We consider, for each $t>0$,
	the functions
	\begin{itemize}
		\item $\displaystyle m_{\varepsilon,s}^{\ell}(t)=\sum_{j=0, \ j\in \N_s}^\ell  \varepsilon_j \mathfrak{b}\left({t\over{2^{j-1}}}\right)$, \quad
		$s=0,1,2,3$, \quad $\ell\in\Bbb N$ and  $\varepsilon=(\varepsilon_j)_{j=0}^\ell \in\{-1,1\}^{\ell+1}$; \\
		
		\item $\displaystyle M_\ell(t) =\sum_{j=0}^\ell {{2^{j\gamma}}\over{(t+1)^{\gamma}}} \mathfrak{a}\left({t\over{2^{j-1}}}\right)$, \quad
		$\ell \in \N$; \\
		
		\item $\displaystyle R_\ell(t)={\phi}/M_\ell(t)$, \quad where $\phi$ is as in Lemma~\ref{Lem:6.1}; \\
		
		\item $\displaystyle R(t)=\left({t\over{t+1}}\right)^{\gamma}$.
	\end{itemize}
	Then, the spectral multipliers $m_{\varepsilon,s}^{\ell}(\mathcal{L}_{\alpha,\beta})$, $M_\ell(\mathcal{L}_{\alpha,\beta})$,
	$R_\ell(\mathcal{L}_{\alpha,\beta})$ and $R(\mathcal{L}_{\alpha,\beta})$
	define bounded operators in $L^p_w(0,\pi)$. Moreover,
	$$\sup_{s, \ell, \varepsilon}\|m_{\varepsilon,s}^{\ell}(\mathcal{L}_{\alpha,\beta}) \|_{ L_{w}^p(0,\pi)\rightarrow  L_{w}^p(0,\pi)}<\infty,$$
	and
	$$\sup_{\ell} \Big( \|M_\ell(\mathcal{L}_{\alpha,\beta}) \|_{ L_{w}^p(0,\pi)\rightarrow  L_{w}^p(0,\pi)}
	    						+ \|R_\ell(\mathcal{L}_{\alpha,\beta}) \|_{ L_{w}^p(0,\pi)\rightarrow  L_{w}^p(0,\pi)}\Big) <\infty.$$
\end{Lem}

\begin{proof}
	Again, by Proposition \ref{propM1}, it suffices to take into account that, for every $k\in\Bbb N$ there exists $C>0$ for which
	$$\sup_{t\in (0,\infty)}\Big|t^k{{d^k}\over{dt^k}}m_{\varepsilon,s}^{\ell}(t)\Big|\leq C,$$
	where $C>0$ does not depend on $s$, $\ell$ or $\varepsilon$.
	Also, $M_\ell=m_{\varepsilon}^{\ell}$ in Lemma~\ref{Lem:6.1}, for $\varepsilon=(1)_{j=0}^\ell $. Finally, for every $k\in\Bbb N$, there exists $C>0$ such that
	$$\sup_{t\geq \lambda_0^{\alpha,\beta}/2}\Big|t^k{{d^k}\over{dt^k}}{1\over{M_\ell(t)}}\Big|
		\leq C,$$
	where $C>0$ does not depend on $\ell$.
\end{proof}

\begin{proof}[Proof of Theorem \ref{Th3}; the case of $F_{\alpha,\beta}^{\gamma,2,p(\cdot)}(0,\pi) \subseteq H_{\alpha,\beta}^{\gamma,p(\cdot)}(0,\pi)$]
	Suppose that $s\in\{0,1,2,3\}$ and $n\in\Bbb N\backslash\{0\}$.  We define
	$$g_{s,\ell}^{\alpha,\beta}(f)
		=\sum_{j=0, \ j\in \N_s}^\ell2^{j\gamma} {\Phi}_j^{\alpha,\beta}(f), \quad \ell\in\Bbb N\ \mbox{and}\ f\in L^1(0,\pi).$$
	There exists at most an unique $j_n\in{\Bbb N}_s$ such that
	${\lambda}_n^{\alpha,\beta}\in[2^{j_n-2},2^{j_n})$.
	Hence,
	$$\mathfrak{b}\left({{{\lambda}_n^{\alpha,\beta}}\over{2^{j_n-1}}}\right)=1 \qquad \text{and} \qquad
	    \mathfrak{b}\left({{{\lambda}_n^{\alpha,\beta}}\over{2^{j-1}}}\right)
	    		=\mathfrak{a}\left({{{\lambda}_n^{\alpha,\beta}}\over{2^{j-1}}}\right)=0, \quad j\in{\Bbb N}_s, \ j\neq j_n.$$
	Observe that $m_{\varepsilon,s}^{\ell}({\lambda}_n^{\alpha,\beta})=\varepsilon_{j_n}$, provided that $j_n\leq \ell$, and
	$m_{\varepsilon,s}^{\ell}({\lambda}_n^{\alpha,\beta})=0$, otherwise.
	We can write
	\begin{align*}
		g_{s,\ell}^{\alpha,\beta}(f)
			= & \sum_{j=0, \ j\in \N_s}^\ell2^{j\gamma}
				 \sum_{n=0}^{\infty} \mathfrak{a}\left({{{\lambda}_n^{\alpha,\beta}}\over{2^{j-1}}}\right) c_n^{\alpha,\beta}(f){\phi}_n^{\alpha,\beta}
			=  \sum_{n=0}^{\infty} a_n \ c_n^{\alpha,\beta}(f){\phi}_n^{\alpha,\beta}, \quad f\in L^1(0,\pi).
	\end{align*}
	where
	$a_n=2^{j_n\gamma}\mathfrak{a}\left({\lambda}_n^{\alpha,\beta}/2^{j_n-1}\right)$, if  $j_n\leq \ell$, and $a_n=0$, otherwise.
	Note that the above serie is actually a finite sum.
	Also, we have that
	\begin{align*}
		m_{\varepsilon,s}^{\ell}(\mathcal{L}_{\alpha,\beta}) g_{s,\ell}^{\alpha,\beta}(f)
			& = \sum_{n=0}^{\infty}m_{\varepsilon,s}^{\ell}\left({\lambda}_n^{\alpha,\beta}\right) a_n \ c_n^{\alpha,\beta}(f) {\phi}_n^{\alpha,\beta}
			   = \sum_{n=0}^{\infty}{\varepsilon}_{j_n} a_n \ c_n^{\alpha,\beta}(f){\phi}_n^{\alpha,\beta}  \\
			& = \sum_{j=0, \ j\in \N_s}^\ell2^{j\gamma}{\varepsilon}_j\sum_{n=0}^{\infty}
					\mathfrak{a}\left({{{\lambda}_n^{\alpha,\beta}}\over{2^{j-1}}}\right) c_n^{\alpha,\beta}(f){\phi}_n^{\alpha,\beta}
			 = \sum_{j=0, \ j\in \N_s}^\ell2^{j\gamma}{\varepsilon}_j{\Phi}_j^{\alpha,\beta}(f).
 	\end{align*}
  	Then,
  	$$m_{\varepsilon,s}^{\ell}(\mathcal{L}_{\alpha,\beta}) m_{\varepsilon,s}^{\ell}(\mathcal{L}_{\alpha,\beta}) g_{s,\ell}^{\alpha,\beta}(f)
		=\sum_{n=0}^{\infty} a_n \ c_n^{\alpha,\beta}(f){\phi}_n^{\alpha,\beta}
		=g_{s,\ell}^{\alpha,\beta}(f).$$
	Assume that $1<p<\infty$ and $w\in A_p(0,\infty)$. From Lemma~\ref{Lem:6.2} we get
	\begin{align*}
		\|g_{s,\ell}^{\alpha,\beta}(f)\|_{L_{w}^p(0,\pi)}
			& \leq C\| m_{\varepsilon,s}^{\ell}(\mathcal{L}_{\alpha,\beta}) g_{s,\ell}^{\alpha,\beta}(f)\|_{L_{w}^p(0,\pi)}
			   \leq C \Big\|\sum_{j=0, \ j\in \N_s}^\ell2^{j\gamma}{\varepsilon}_j{\Phi}_j^{\alpha,\beta}(f)\Big\|_{L_{w}^p(0,\pi)}, \quad f\in L_{w}^p(0,\pi),
	\end{align*}
	where $C>0$ does not depend on $\varepsilon$ or $\ell$.
	By using Khintchine's inequality argument we obtain
	$$\|g_{s,\ell}^{\alpha,\beta}(f)\|_{L_{w}^p(0,\pi)}
		\leq C \Big\|\sum_{j=0, \ j\in \N_s}^\ell(2^{j\gamma}{\varepsilon}_j|{\Phi}_j^{\alpha,\beta}(f)|)^2)^{1/2}\Big\|_{L_{w}^p(0,\pi)}, \quad f\in L_{w}^p(0,\pi),$$
	where $C>0$ does not depend on $\ell$.
	According to \cite[Theorem 1.3]{CFMP},
	$$\|g_{s,\ell}^{\alpha,\beta}(f)\|_{L^{p(\cdot)}(0,\pi)}
		\leq C \Big\|\sum_{j=0, \ j\in \N_s}^\ell(2^{j\gamma}{\varepsilon}_j|{\Phi}_j^{\alpha,\beta}(f)|)^2)^{1/2}\Big\|_{L^{p(\cdot)}(0,\pi)},
		\quad f\in S_{\alpha,\beta},$$
	where $C>0$ does not depend on $\ell$.
	As in the proof of the first inclusion we obtain
	\begin{align}\label{W3}
		\|g_{s,\ell}^{\alpha,\beta}(f)\|_{L^{p(\cdot)}(0,\pi)}
			\leq C \Big\|\sum_{j=0, \ j\in \N_s}^\ell(2^{j\gamma}{\varepsilon}_j|{\Phi}_j^{\alpha,\beta}(f)|)^2)^{1/2}\Big\|_{L^{p(\cdot)}(0,\pi)},
			\quad f\in L^{p(\cdot)}(0,\pi).
	\end{align}
	According to (\ref{W0}) we have that, for every $f\in L^1(0,\pi)$,
	\begin{align*}
		& \sum_{n=0}^{\infty}M_\ell({\lambda}_n^{\alpha,\beta})({\lambda}_n^{\alpha,\beta}+1)^{\gamma}c_n^{\alpha,\beta}(f){\phi}_n^{\alpha,\beta}
		 = \sum_{n=0,{\lambda}_n^{\alpha,\beta}\leq 2^\ell }^{\infty}{\varepsilon}_{j_n}c_n^{\alpha,\beta}(f){\phi}_n^{\alpha,\beta}
		 = \sum_{j=0}^\ell 2^{j\gamma}{\Phi}_j^{\alpha,\beta}(f).
	\end{align*}
	By using (\ref{W3}), Lemma~\ref{Lem:6.2} and \cite[Theorem 1.3]{CFMP} we can write
	\begin{align}\label{W4}
    		& \Big\|\sum_{n=0,{\lambda}_n^{\alpha,\beta}\leq 2^\ell }^{\infty}({\lambda}_n^{\alpha,\beta})^{\gamma}
			c_n^{\alpha,\beta}(f){\phi}_n^{\alpha,\beta}\Big\|_{L^{p(\cdot)}(0,\pi)} \nonumber \\
    		& \qquad = \Big\|\sum_{n=0}^{\infty}R({\lambda}_n^{\alpha,\beta})({\lambda}_n^{\alpha,\beta}+1)^{\gamma}
	R_\ell({\lambda}_n^{\alpha,\beta})M_\ell({\lambda}_n^{\alpha,\beta})c_n^{\alpha,\beta}(f){\phi}_n^{\alpha,\beta} \Big\|_{L^{p(\cdot)}(0,\pi)} \nonumber \\
    		& \qquad \leq C \Big\|\sum_{j=0}^\ell 2^{j\gamma}{\Phi}_j^{\alpha,\beta}(f)\Big\|_{L^{p(\cdot)}(0,\pi)} \nonumber \\
    		& \qquad \leq C\Big(\sum_{s=0}^3 \Big\|\sum_{j=0, \ j\in \N_s}^\ell2^{j\gamma}{\Phi}_j^{\alpha,\beta}(f)\Big\|_{L^{p(\cdot)}(0,\pi)}
			+\|{\Phi}_0^{\alpha,\beta}\|_{L^{p(\cdot)}(0,\pi)}\Big) \nonumber \\
    		& \qquad \leq C\Big(\sum_{s=0}^3\Big\|\Big(\sum_{j=0, \ j\in \N_s}^\ell(2^{j\gamma}|{\Phi}_j^{\alpha,\beta}(f)|)^2\Big)^{1/2}\Big\|_{L^{p(\cdot)}(0,\pi)}
			+\|{\Phi}_0^{\alpha,\beta}\|_{L^{p(\cdot)}(0,\pi)}\Big) \nonumber \\
    		& \qquad \leq C\Big\|\Big(\sum_{j=0}^\ell (2^{j\gamma}|{\Phi}_j^{\alpha,\beta}(f)|)^2\Big)^{1/2}\Big\|_{L^{p(\cdot)}(0,\pi)}, \quad f\in L^{p(\cdot)}(0,\pi).
	\end{align}
  	Suppose now that $f=\displaystyle\sum_{n=m}^\ell c_n^{\alpha,\beta}(f){\phi}_n^{\alpha,\beta}$, where $m,\ell\in\Bbb N$, $m\leq \ell$.
	Since $\supp \mathfrak{a}\subseteq[1/2,2]$, we have that
	\begin{align*}
		{\Phi}_j^{\alpha,\beta}
			& = \sum_{n=0}^{\infty} \mathfrak{a}\left({{{\lambda}_n^{\alpha,\beta}}\over{2^{j-1}}}\right)  c_n^{\alpha,\beta}(f){\phi}_n^{\alpha,\beta}
			 = \sum_{n=m}^\ell \mathfrak{a}\left({{{\lambda}_n^{\alpha,\beta}}\over{2^{j-1}}}\right) c_n^{\alpha,\beta}(f){\phi}_n^{\alpha,\beta}
			 =0,
  	\end{align*}
  	provided that $j>2+\log_2 \ell$ or $j<\log_2 m$. Then, from (\ref{W4}) we deduce that
	\begin{align}\label{W5}
		& \Big\|\sum_{n=m}^\ell ({\lambda}_n^{\alpha,\beta})^{\gamma}c_n^{\alpha,\beta}(f){\phi}_n^{\alpha,\beta}\Big\|_{L^{p(\cdot)}(0,\pi)}
		 \leq C\Big\|\Big(\sum_{j=\log_2 m}^{2+\log_2 \ell}(2^{j\gamma}|{\Phi}_j^{\alpha,\beta}(f)|)^2\Big)^{1/2}\Big\|_{L^{p(\cdot)}(0,\pi)}.
  	\end{align}
  	Let $f\in F_{\alpha,\beta}^{\gamma,2,p(\cdot)}(0,\pi)$.
	By (\ref{W5}), the series
	$\displaystyle\sum_{n=m}^\ell ({\lambda}_n^{\alpha,\beta})^{\gamma}c_n^{\alpha,\beta}(f){\phi}_n^{\alpha,\beta}$ converges in
	$L^{p(\cdot)}(0,\pi)$. Hence, $f\in H_{\alpha,\beta}^{\gamma,p(\cdot)}(0,\pi)$ and by (\ref{W4}) and Proposition \ref{Prop:M5}, we conclude that
  	$$\|f\|_{H_{\alpha,\beta}^{\gamma,p(\cdot)}(0,\pi)}
		\leq C\|f\|_{F_{\alpha,\beta}^{\gamma,2,p(\cdot)}(0,\pi)}.$$
\end{proof}

\noindent{\bf Acknowledgement}: the authors would like to thank Professor L. Diening for explaining us the extension of the variable exponent  $p\in{\mathcal B}(0,\pi)$ as a function in ${\mathcal B}(\Bbb R)$.


\end{document}